\documentclass{article}
\usepackage{amsmath,amsthm,amssymb,stmaryrd}

\newtheorem{theorem}{Theorem}[section]
\newtheorem{proposition}[theorem]{Proposition}
\newtheorem{lemma}[theorem]{Lemma}
\newtheorem{corollary}[theorem]{Corollary}
\newtheorem{remark}[theorem]{Remark}
\theoremstyle{definition}

\newtheorem{definition}[theorem]{Definition}

\newcommand{\term}{\mathcal{T}}
\newcommand{\powset}[1]{\mathcal{P}(#1)}
\newcommand{\pcaa}{\mathcal{A}}
\newcommand{\mv}[2]{\operatorname{mv}(#1, #2)}
\newcommand{\vip}[1]{V_0^{\mathrm{f}}(#1)}
\newcommand{\vg}[1]{V_0^\Gamma(#1)}
\newcommand{\vga}{\vg{\pcaa}}
\newcommand{\vgt}{\vg{\term}}
\newcommand{\vt}{V(\term)}
\newcommand{\vdzero}{\Vdash_0}
\newcommand{\vdzg}{\Vdash_0^\Gamma}
\newcommand{\vdone}{\Vdash_1}
\newcommand{\stab}[2]{\operatorname{Stab}_{#1}(#2)}
\newcommand{\cip}{C^{\mathrm{f}}}
\newcommand{\cg}{C^\Gamma}
\newcommand{\pcak}{\mathbf{k}}
\newcommand{\pcas}{\mathbf{s}}

\newcommand{\lmim}[1]{\operatorname{RED}(#1)}
\newcommand{\lmimn}[2]{\operatorname{RED}_{#1}(#2)}
\newcommand{\mvnnn}{\mv{\mathbb{N}^{\mathbb{N}}}{\mathbb{N}}}
\newcommand{\zf}{\mathbf{ZF}}
\newcommand{\izf}{\mathbf{IZF}}
\newcommand{\izfr}{\mathbf{IZF}_R}
\newcommand{\czf}{\mathbf{CZF}}
\newcommand{\czfm}{\mathbf{CZF}^-}
\newcommand{\czfexp}{\mathbf{CZF}_{\mathcal{E}}}
\newcommand{\czfpow}{\mathbf{CZF}_{\mathcal{P}}}
\newcommand{\cst}{\mathbf{CST}}
\newcommand{\cstm}{\mathbf{CST}^-}
\newcommand{\rdc}{\mathbf{RDC}}
\newcommand{\zfc}{\mathbf{ZFC}}

\newcommand{\powip}{\mathcal{P}_\pcaa^\mathrm{f}}

\title{CZF does not have the Existence Property}
\author{Andrew W Swan}

\begin{document}
\maketitle

\begin{abstract}
Constructive theories usually have interesting metamathematical
properties where explicit witnesses can be extracted from proofs of
existential sentences. For relational theories, probably the most
natural of these is the \emph{existence property}, EP, sometimes
referred to as the \emph{set existence property}. This states that
whenever $(\exists x)\,\phi(x)$ is provable, there is a formula
$\chi(x)$ such that $(\exists ! x)\;\phi(x) \wedge \chi(x)$ is
provable. It has been known since the 80's that EP holds for some
intuitionistic set theories and yet fails for $\izf$. Despite this, it
has remained open until now whether EP holds for the most well
known constructive set theory, $\czf$.  In this paper we show that EP
fails for $\czf$.
\end{abstract}

\section{Introduction}
\subsection{Existence Properties}
\label{intcst}
Constructive theories are known for having metamathematical properties
that are often not shared by stronger classical theories such as
$\zfc$. The principles below are amongst the most well known of
these properties.

Constructive mathematicians choose to interpret disjunctions and
existential quantifiers more strictly than classical mathematicians.
For the constructive mathematician, in order to know the disjunction
$\phi \vee \psi$, one must either know $\phi$ or know $\psi$. They
therefore often expect their formal theories to have the following
property:
\begin{definition}
A theory, $T$ has the \emph{disjunction property} (DP) if whenever
$T \vdash \phi \vee \psi$, either $T \vdash \phi$ or $T \vdash \psi$.
\end{definition}

In order to know $(\exists x)\,\phi(x)$, the constructive mathematician
must be able to ``construct'' some witness $a$ such that one knows
$\phi(a)$. We certainly know what it means to construct an element of
$\omega$: we must be able to write down an actual natural number. We
also know what it means to construct a function $\mathbb{N}
\rightarrow \mathbb{N}$: we must be able to able to find (a number
encoding) an algorithm whose graph is that function. Hence the
constructive mathematician expects their formal theories to have the
following properties. In the definitions below we assume that $T$ has
a constant $\omega$ such that $T$ proves that $\omega$ is the natural
numbers and for each $n$ a constant $\overline{n}$ such that $T$
proves $\overline{0}$ is empty and $\overline{n + 1}$ is the successor
of $\overline{n}$. (For any theory that could ``reasonably'' be called
a set theory, there will be at least a conservative extension with
these constants).

\begin{definition}
  $T$ has the \emph{numerical existence property} (NEP) if whenever $T
  \vdash (\exists x \in \omega)\,\phi(x)$ where $\phi(x)$ only has $x$
  free, there is some natural $n$ such that \[T \vdash
  \phi(\overline{n})\]
\end{definition}

\begin{definition}
  $T$ is closed under \emph{Church's Rule} (CR) if whenever $T \vdash
  (\forall x \in \omega)(\exists y \in \omega)\,\phi(x, y)$, where
  $\phi$ only has free variables $x$ and $y$, there is some natural
  $e$ such that
\[
T \vdash (\forall x \in \omega)\,\phi(x, \{\overline{e}\}(x))
\]
(where $\{e\}(x)$ denotes the result of applying the $e$th recursive
function to $x$)
\end{definition}

What it means to construct mathematical objects in general is 
less clear, but a common interpretation of this is that they
should at least be definable, in the sense below.
\begin{definition}
  $T$ has the \emph{existence property} (EP) if whenever $T \vdash
  (\exists x)\,\phi(x)$, where $\phi$ only has free variable $x$, there
  is some formula $\chi(x)$ that only has free variable $x$ such that
\[T \vdash (\exists ! x)\;\phi(x) \wedge \chi(x)\]
\end{definition}

\subsection{$\izf$, $\cst$, and $\czf$}

If one wants a theory with some of the metamathematical properties
appearing in section \ref{intcst} but has no other objections to
classical mathematics, one may be satisfied with the theory $\izf$,
which can be regarded as ``$\mathbf{ZF}$ without excluded middle.''

\begin{definition}
$\izf$ is the theory with (intuitionistic logic and) the following
  axioms:
\begin{enumerate}
\item Extensionality
\item Separation
\item Pairing
\item Union
\item Infinity
\item Power Set
\item $\in$-induction
\item Collection
\end{enumerate}
\end{definition}

Collection is the following schema (where $\phi$ may have other free
variables in addition to $x$ and $y$):
\[
(\forall x \in a)(\exists y)\,\phi(x, y) \;\rightarrow\; (\exists z)(\forall
x \in a)(\exists y \in z)\,\phi(x, y)
\]
Compare this with the schema (equivalent in $\zf$) Replacement:
\[
(\forall x \in a)(\exists ! y)\,\phi(x, y) \;\rightarrow\; (\exists z)(\forall
x \in a)(\exists y \in z)\,\phi(x, y)
\]

\begin{definition}
$\izfr$ is the set theory with the axioms of $\izf$
  except that it has Replacement instead of Collection.
\end{definition}

For technical reasons, in this paper we will in fact consider a
variation of the more usual definitions above that have bounded
quantifiers as primitives. That is, for every formula,
$\phi(x, z_1,\ldots,z_n)$, with free variables amongst $x,z_1,\ldots,z_n$
there are $n+1$-ary predicates $(\forall x \in
y)\,\phi(x,z_1,\ldots,z_n)$ and $(\exists x \in
y)\,\phi(x,z_1,\ldots,z_n)$ together with appropriate axioms for
bounded universal and bounded existential quantification. Note that
this is a conservative extension of the usual definitions above, and
therefore the properties mentioned in section \ref{intcst} hold in
these versions if and only if they hold in the usual versions.

$\izf$ is extremely powerful. In fact Friedman showed in
\cite{friedman73a} that it has the same consistency strength as
$\zf$. On the other hand, $\izf$ has most of the existence properties
we saw earlier.

Often one may be doing mathematics constructively for philosophical
reasons. One may be an intuitionist: one believes mathematical objects
only exist if they can be ``mentally constructed.'' One may be a
predicativist: one believes that a mathematical object cannot be
constructed until it is defined predicatively - that is without
quantifiers whose range includes the object being constructed.
In this case one needs to ensure that the axioms of the set theory are
constructively justified. There are (at least) two ways to go about
this:

\begin{enumerate}
\item Directly justify each axiom as ``true'' with philosophical
  reasoning
\item Find another theory that already has a strong constructive
  foundation and interpret your set theory into it
\end{enumerate}

Myhill in \cite{Myhill75} took the first approach, introducing the
following theories. Both of these are over a three sorted language
with sorts for numbers, sets, and partial functions.

\begin{definition}
  $\cstm$ is the theory with (intuitionistic logic and) the following
  axioms:
  \begin{enumerate}
  \item Extensionality (for sets)
  \item Bounded Separation (that is, separation for formulae where
    every quantifier is bounded)
  \item Pairing
  \item Union
  \item Exponentiation (that is, given any sets $A$ and $B$ there is a
    set containing precisely the functions $f:A \rightarrow B$)
  \item Replacement
  \item Axioms of Heyting Arithmetic for the number sort
  \end{enumerate}
\end{definition}

\begin{definition}
$\cst$ is the theory $\cstm$ together with relativised dependent
  choices $\rdc$.
\end{definition}

In particular Myhill rejected the power set axiom in favour of the
weaker exponentiation axiom because of the more predicative nature of
exponentiation. He chose bounded separation over full separation for
the same reason.

$\czf$ arose via the second approach in \cite{Aczel78} where Aczel
showed that set theory can be interpreted into the predicative
Martin-L\"{o}f type theory. Aczel also dropped the three sorted
approach of $\cst$, instead working over the same language as $\zf$.

\begin{definition}
  $\czf$ is the theory with (intuitionistic logic and) the following
  axioms

  \begin{enumerate}
  \item Extensionality
  \item Bounded Separation
  \item Pairing
  \item Union
  \item Strong Infinity
  \item Subset Collection: the schema
    \begin{multline*}
    (\exists c) (\forall u) ((\forall x \in a) (\exists y \in b)\, \psi(x, y, u)\;
    \rightarrow \\
    (\exists d \in c) ((\forall x \in a) (\exists y \in d)\, \psi(x, y, u)
    \;\wedge\; (\forall y \in d) (\exists x \in a)\, \psi(x, y, u)))
    \end{multline*}
  \item $\in$-induction
  \item Strong Collection: the schema
    \begin{multline*}
      (\forall x \in a) (\exists y) \,\phi(x, y) \;\rightarrow\; \\
      (\exists b)((\forall
      x \in a) (\exists y \in b)\, \phi(x, y) \;\wedge\; (\forall y \in b) (\exists
      x \in a)\, \phi(x, y))
    \end{multline*}
  \end{enumerate}
\end{definition}

As with $\izf$, in this paper we will in fact work with the
conservative extension of $\czf$ that also has predicates and axioms
for bounded quantification.

Subset collection implies exponentiation and is implied by power set
(see \cite{aczelrathjen}) and can be seen as an ``artifact'' of the
interpretation of set theory into type theory.
As an alternative to subset collection, one may instead assume the
equivalent fullness axiom. Given sets $A$ and $B$, define $\mv{A}{B}$
to be the class of multivalued functions as
\[
\mv{A}{B} := \{ R \subseteq A \times B \;|\; (\forall a \in A)(\exists
b \in B)\,(a, b) \in R \}
\]
The fullness axiom can then be stated as follows
\begin{displaymath}
(\forall A, B)(\exists C \subseteq \mv{A}{B})(\forall R)\;R \in
  \mv{A}{B} \,\rightarrow\, (\exists S \in C)\,(S \subseteq R)
\end{displaymath}
(For a more detailed discussion of the fullness axiom see
\cite{aczelrathjen}.)

One can see that the fullness axiom asserts the existence of sets for
which there is no apparent definition. We will prove that for the case
$A = \mathbb{N}^{\mathbb{N}}, B = \mathbb{N}$, there is no definable
$C$.

$\czf$ is stronger than $\cstm$ in two respects: replacement has been
strengthened to strong collection and exponentiation has been
strengthened to subset collection.

$\czf$ is regarded today as the standard set theory for formalising
constructive mathematics. This is because it is constructively valid
because of its interpretation into type theory and yet can be used to
prove mathematically interesting results that do not hold in weaker
theories. For example, in \cite{LubarskyRathjen} Lubarsky and Rathjen
showed that the theory $\czfexp$ that has only exponentiation in place
of subset collection does not prove that the Dedekind reals form a
set, whereas $\czf$ does prove that the Dedekind reals form a set.

\subsection{Existence Properties of these Set Theories}
\label{existingwork}
The properties DP, NEP, and CR work extremely well as
characterisations of constructive formal theories. None can hold for
consistent recursively axiomatisable theories that have excluded
middle, but on the other hand they hold for a rich variety of
constructive theories.

In \cite{Myhill73} Friedman and Myhill showed that $\izfr$ (that is,
$\izf$ with replacement instead of collection), has the existence
property. In \cite{Myhill75}, Myhill showed the set theory $\cstm$
also has EP and also that both $\cstm$ and $\cst$ have DP and NEP,
leaving open whether $\cst$ has EP. In \cite{FriedmanScedrov83}
Friedman and \u{S}\u{c}edrov showed that $\izfr + \mathbf{RDC}$ has
EP, establishing that even set theories with choice principles can
have EP.

Beeson then developed $q$-realizability, allowing him to show in
\cite{beeson85} that NEP, DP, and CR hold for $\izf$ and $\izf +
\mathbf{RDC}$.  Rathjen developed realizability with truth based
partly on Beeson's methods to show in \cite{rathjen05} and
\cite{rathjen08} that DP, NEP, CR and other properties hold for a wide
variety of intuitionistic set theories including $\czf$, $\czf +
\mathbf{REA}$, $\izf$, $\izf + \mathbf{REA}$ with any combination of
the axioms $\mathbf{MP}$, $\mathbf{AC_\omega}$, $\mathbf{DC}$,
$\mathbf{RDC}$ and $\mathbf{PAx}$.

One can see that EP does not work so well as a characterisation
of constructive theories as the other properties we have seen. As
remarked in \cite{rathjen05} EP can hold for classical
theories, even extensions of $\mathbf{ZFC}$. On the other hand,
Friedman and \u{S}\u{c}edrov showed in \cite{FriedmanScedrov85} that
$\izf$ \emph{does not} have EP.

Friedman and \u{S}\u{c}edrov's proof that EP fails for $\izf$ makes
use of full separation and collection. Since $\izfr$ does have EP, it
might seem reasonable to think that collection is responsible for the
failure of EP and the use of full separation is only
incidental. However due to recent work by Rathjen, this turns out not
to be the case. Set theories with collection but only bounded
separation can have EP.

In \cite{rathjen11} Rathjen defined the following two variations on
EP,
\begin{definition}
  \begin{enumerate}
  \item $T$ has the \emph{weak existence property}, wEP, if whenever
    $\phi(x)$ is a formula with $x$ the only free variable and 
    \[ T \vdash (\exists x)\, \phi(x) \]
    there is some formula $\chi(x)$ having at most the free variable
    $x$ such that
    \begin{eqnarray*}
      T & \vdash & (\exists ! x)\,\chi(x) \\
      T & \vdash & (\forall x)\,(\chi(x) \rightarrow (\exists u)\, u \in x) \\
      T & \vdash & (\forall x)\,(\chi(x) \rightarrow (\forall u \in x)\,
        \phi(u)) \\
    \end{eqnarray*}
  \item $T$ has the \emph{uniform weak existence property}, uwEP, if
    whenever $\phi(u, x)$ is a formula having at most $x$ and $u$ as
    free variables and
    \[ T \vdash (\forall u) (\exists x)\, \phi(u, x) \]
    there is some formula $\chi(u, x)$ having at most the free variables
    $u, x$ such that
    \begin{eqnarray*}
      T & \vdash & (\forall u) (\exists ! x)\, \chi(u, x) \\
      T & \vdash & (\forall u) (\forall x)\,(\chi(u, x) \rightarrow (\exists
        z)\, z \in x) \\
      T & \vdash & (\forall u) (\forall x)\,(\chi(u, x) \rightarrow (\forall
        z \in x)\, \phi(u, z)) \\
    \end{eqnarray*}
  \end{enumerate}
\end{definition}

As remarked in \cite{rathjen11}, by analysing Friedman and
\u{S}\u{c}edrov's proof in \cite{FriedmanScedrov85} one can see that
$\izf$ doesn't even have wEP. On the other hand any extension of
$\mathbf{ZF}$ has uwEP - consider $V_\alpha$ where $\alpha$ is the
least ordinal such that $V_\alpha$ contains a witness.

In \cite{rathjen11}, Rathjen refers to the theories $\czfm$, $\czfexp$
and $\czfpow$. $\czfm$ is $\czf$ without subset collection. $\czfexp$
is $\czfm$ with the exponentiation axiom. $\czfpow$ is $\czf$ together
with the power set axiom. All three of these theories have strong
collection, and yet Rathjen shows in \cite{rathjen11} that all three
have uwEP (and hence wEP). In that paper he refers to a paper in
preparation where he will show by using this result together with
ordinal analysis that these three theories in fact have EP.

$\czfpow$, which has EP, is simply $\izf$ with bounded separation in
place of full separation, so the use of full separation in Friedman
and \u{S}\u{c}edrov's proof must be essential. Furthermore, $\czf$
lies between $\czfexp$ and $\czfpow$, two theories both satisfying EP
and uwEP.

However, due to problems defining witnesses for the fullness axiom,
these proofs do not apply to $\czf$ itself. Rathjen goes so far as to
conjecture in \cite{rathjen11} that $\czf$ does not even have
wEP. In this paper we prove that this conjecture is correct. $\czf$
does not have wEP, and the fullness axiom is responsible.

\subsection{Pcas}
When defining realizability, one usually starts with a partial
combinatory algebra (pca).
\begin{definition}
  A pca, $\pcaa$ is a set $|\pcaa|$ together with a partial binary
  operation, $\cdot : |\pcaa| \times |\pcaa| \rightharpoondown |\pcaa|$
  referred to as \emph{application}, and distinguished elements,
  $\pcas$ and $\pcak$ satisfying the axioms below. Below, and
  throughout the paper we will write $\cdot (a, b)$ as $(a . b)$, or
  simply $a b$, and follow the convention that application is
  \emph{left associative}. That is, $a_1 a_2 \ldots a_n$ means
  $(\ldots((a_1 . a_2) . a_3)\ldots). a_n)$.

\begin{enumerate}
\item $\pcas \neq \pcak$
\item for all $a, b \in \pcaa$, $\pcak a b \simeq a$
\item for all $a, b \in \pcaa$, $\pcas a \downarrow, \pcas a b
  \downarrow$
\item for all $a, b, c \in \pcaa$, $\pcas a b c \simeq a c (b c)$
\end{enumerate}
\end{definition}

Recall the following from, for example \cite{vanoosten} or
\cite{beeson85}.

\begin{definition}
\label{termoveradef}
Given a pca, $\pcaa$, we define \emph{terms} over
$\pcaa$ inductively as follows

\begin{enumerate}
\item There is a countable supply of free variables, $x_i$, each of
  which is a term.
\item Each element, $a$ of $\pcaa$ is a term.
\item If $s$ and $t$ are terms, then the ordered pair, $\langle s, t
  \rangle$ is also a term. We write this as $(s.t)$.
\end{enumerate}

We say that a term is \emph{closed} if it contains no free variables.
\end{definition}

As before, we will usually write the term $(s.t)$ just as $s t$ and
follow the convention of left associativity.

\begin{definition}
\label{denotedef}
We define inductively what it means for a closed term, $s$, to
\emph{denote} $a \in \pcaa$
\begin{enumerate}
\item If $a' \in \pcaa$, then $a'$ denotes $a$ if and only if $a = a'$
\item $(s'.s'')$ denotes $a$ if and only if there are $a', a'' \in \pcaa$
  such that $s'$ denotes $a'$, $s''$ denotes $a''$, and $a'.a'' \simeq a$.
\end{enumerate}

If $t$ is a closed term and there is an $a \in \pcaa$ such that $t$
denotes $a$, we write $t \downarrow$.

If $t(x_1, \ldots, x_n)$ is an open term with free variables amongst
$x_1, \ldots, x_n$, we write $t \downarrow$ to mean that for every
$a_1, \ldots, a_n \in \pcaa$, $t(a_1, \ldots, a_n) \downarrow$. If
$t(x_1,\ldots,x_n)$ and $t'(x_1,\ldots,x_n)$ are open terms with free
variables amongst $x_1, \ldots, x_n$, we write $t(x_1,\ldots,x_n)
\simeq t'(x_1,\ldots,x_n)$ to mean that for all $a_1,\ldots,a_n \in
\pcaa$, $t(a_1,\ldots,a_n) \simeq t'(a_1,\ldots,a_n)$. That is,
$t(a_1,\ldots,a_n)$ denotes if and only if $t'(a_1,\ldots,a_n)$
denotes, and if they denote then they denote the same element of $\pcaa$.
\end{definition}

\begin{proposition}
  For any term, $t(x_1,\ldots,x_n)$, over $\pcaa$ (that may contain
  $x_1,\ldots,x_n$ as free variables and possibly more free variables
  in addition to these) there is a term $t^\ast$ that does not contain
  the free variables $x_1,\ldots,x_n$ and such that for all
  $a_1,\ldots,a_n \in \pcaa$, $t^\ast a_1 \ldots a_{n - 1} \downarrow$
  and
  \[
  t^\ast a_1 \ldots a_n \simeq t(a_1, \ldots, a_n)
  \]
\end{proposition}

We will write $t^\ast$ as $\lambda (x_1,\ldots,x_n).t(x_1,\ldots,x_n)$.

\begin{proposition}
  For any $\pcaa$ there are $y, y' \in \pcaa$ such that for all $f \in
  \pcaa$,
  \begin{enumerate}
  \item $y f \simeq f (y f)$
  \item $y' f \downarrow$ and for all $e \in \pcaa$, $(y' f) e \simeq
    f (y' f) e$
  \end{enumerate}
\end{proposition}

One can use this to construct pairing and projection operators that we
will refer to as $\mathbf{p}$, $\mathbf{p}_0$ and $\mathbf{p}_1$. We
will write $(e)_i$ to mean $\mathbf{p}_i e$ for $i = 0, 1$. One can
further define numerals that we will denote $\underline{n}$ for each
$n \in \omega$. All recursive functions can then be represented. See
chapter 6 of \cite{beeson85} or chapter 1 of \cite{vanoosten} for
details.

\subsection{The Model $V(\pcaa)$}
\label{vaint}

Realizability is one of the main tools in the study of intuitionistic
theories and was used for many well known results including those
mentioned in section \ref{existingwork}.

The variants of realizability used here have their roots in
\cite{kreiseltroelstra}, where Kriesel and Trolestra adapted Kleene's
realizability from \cite{kleene45} to work with second order
arithmetic. This was later adapted and used by Friedman in
\cite{friedman73}, by Myhill in \cite{Myhill73} and \cite{Myhill75}
and by Beeson in \cite{BeesonScedrov} and \cite{beeson85}. In
\cite{mccarty} and \cite{mccarty86}, McCarty adapted Beeson's
definition to work for set theories with extensionality.
The definition below is the variation introduced by Rathjen in
\cite{rathjen06}, where bounded quantifiers are kept separate in the
definition.

Note that while realizability models intuitionistic logic, it takes
place in a ``background universe'' that may include classical
logic. In \cite{mccarty}, McCarty carries out the realizability
construction within $\zf$, although he remarks that in fact $\izf$ is
sufficient. In this paper, we work within a background universe of
$\zfc$.

Following \cite{rathjen06}, we start by defining the class $V(\pcaa)$
using inductive definitions. Recall (for example from
\cite{aczelrathjen}, section 5) that an inductive definition is a
class, $\Phi$, of ordered pairs and we write $\frac{X}{a} \Phi$ to
mean $\langle X, a \rangle \in \Phi$. Then we have the following
theorem.

\begin{theorem}
  For any inductive definition, $\Phi$, there is a class $I(\Phi)$
  which is the smallest class, $Y$, such that whenever $X \subseteq Y$
  and $\frac{X}{a}\Phi$ we have $a \in Y$.
\end{theorem}

\begin{proof}
  This is theorem 5.1 in \cite{aczelrathjen}.
\end{proof}

We define $V(\pcaa)$ using the inductive definition
\[
\frac{X}{a} \Phi \text{ iff } a \subseteq |\pcaa| \times X
\]

That is, explicitly $V(\pcaa)$ is the smallest class $Y$ such that
\[
\powset{|\pcaa| \times Y} \subseteq Y
\]

We then introduce a relation $\Vdash$ between $\pcaa$ and formulae
with parameters over $V(\pcaa)$. The first two lines of the definition are
defined simultaneously by $\in$-induction and the remaining lines
allow one to inductively define realizability for any sentence,
$\phi$.

\begin{eqnarray*}
e \Vdash a \in b & \text{iff} & \text{there is } \langle  (e)_0, c \rangle \in
b \text{ such that }  (e)_1 \Vdash a = c \\
e \Vdash a = b & \text{iff} & \text{for every } \langle  f,\, c \rangle \in a, 
(e)_0 f \Vdash c \in b \text{ and } \\
& & \qquad \text{for every } \langle  f, c \rangle
\in b,\,  (e)_1 f  \Vdash c \in a \\
e \Vdash \phi \wedge \psi & \text{iff} & (e)_0 \Vdash \phi \text{ and }
(e)_1 \Vdash \psi \\
e \Vdash \phi \vee \psi & \text{iff} & \text{either } (e)_0 =
\underline{0} \text{ and }
(e)_1 \Vdash \phi, \text{ or } (e)_0 = \underline{1} \text{ and } (e)_1 \Vdash
\psi \\
e \Vdash \phi \rightarrow \psi & \text{iff} & f \Vdash \phi \text{
  implies } e f \Vdash \psi  \\
e \Vdash (\exists x \in a)\,\phi(x) & \text{iff} & \text{there is } \langle 
(e)_0, b \rangle \in a \text{ such that } (e)_1 \Vdash \phi(b) \\
e \Vdash (\forall x \in a)\,\phi(x) & \text{iff} & \text{for every } \langle 
f, b \rangle \in a,\, e f \Vdash \phi(b)  \\
e \Vdash (\exists x)\,\phi(x) & \text{iff} & \text{there is } a \in V(\pcaa)
\text{ such that } e \Vdash \phi(a) \\
e \Vdash (\forall x)\,\phi(x) & \text{iff} & \text{for every } a \in V(\pcaa),\,
e \Vdash \phi(a)
\\
e \Vdash \neg \phi & \text{iff} & f \Vdash \phi \text{ is false for
  every } f \in \pcaa
\end{eqnarray*}

If $\phi$ has free variables amongst $x_1, \ldots, x_n$, we write $e
\Vdash \phi$ to mean $e \Vdash (\forall x_1,\ldots, x_n)\,\phi$ (the
\emph{universal closure} of $\phi$).

We write $V(\pcaa) \models \phi$ to mean that there is $e \in \pcaa$
such that $e \Vdash \phi$.

This structure has been defined so that we get soundness for $\izf$ in
the following sense.

\begin{theorem}
Suppose that $\phi$ is a theorem of $\izf$. Then, $V(\pcaa) \models
\phi$.
\end{theorem}

Recall from chapter 5 of \cite{mccarty} that $V(\pcaa)$ has certain
standard representations of the naturals and Baire space.

Define 
\begin{eqnarray*}
\overline{n} & = & \{ \langle \underline{m}, \overline{m} \rangle
\;|\; m < n \} \\
\overline{\omega} & = & \{ \langle \underline{n}, \overline{n} \rangle
\;|\; n \in \omega \}
\end{eqnarray*}
Then $V(\pcaa)$ has a realizer for the statement that
$\overline{\omega}$ is the set of natural numbers. We will also write
$\overline{\omega}$ as $\mathbb{N}$.

Suppose that $f \in \pcaa$ satisfies that for all $n \in \omega$,
there exists $m \in \omega$ such that $f \underline{n} =
\underline{m}$. Then write
\[
\overline{f} = \{ \langle \underline{n}, (\overline{n}, \overline{m})
\rangle \;|\; n, m \in \omega, f \underline{n} = \underline{m} \}
\]
(where we write $(,)$ for $V(\pcaa)$'s internal notion of ordered
pair).

There is a realizer in $V(\pcaa)$ for the statement that the set
of functions from $\omega$ to $\omega$ is precisely
\[
\mathbb{N}^\mathbb{N} := \{ \langle f, \overline{f} \rangle \;|\; (\forall n \in
\omega)(\exists m \in \omega)\,f \underline{n} = \underline{m} \}
\]

\section{Outline of the Proof}
We will show that wEP fails for $\czf$. We do this by first showing
that for any pca, $\pcaa$, we can construct three realizability
models, $V(\pcaa)$, $\vip{\pcaa}$, and $\vg{\pcaa}$. $V(\pcaa)$ is the
usual realizability model of $\izf$ from section
\ref{vaint}. $\vip{\pcaa}$ and $\vg{\pcaa}$ are based on functional sets and symmetric sets respectively and will be described
in detail below.

The heart of the proof is that $\vip{\pcaa}$ and $\vg{\pcaa}$ are
essentially different models and yet both can be embedded into
$V(\pcaa)$ in such a way that realizability is preserved by the
embedding.
$\vip{\pcaa}$ and $\vg{\pcaa}$ must both provide witnesses of
existential statements that are still valid in
$V(\pcaa)$. Definability will imply that these witnesses are
realizably equal in $V(\pcaa)$.

The final step of the proof is to construct a particular pca, $\term$,
based on term models and normal filter $\Gamma$ to show wEP
fails. The cause of this failure will be a simple instance of the
fullness axiom.

\section{The Model $\vg{\pcaa}$}
\subsection{Definitions}
A standard technique for showing the independence of choice principles
in classical set theories is by using \emph{symmetric models}. These
can be seen as boolean valued models where every element is
``symmetric.'' See, for example \cite{bvmodels},
\cite{jechst}, \cite{jechac}, or \cite{kunen} for a detailed
description.  We will construct a realizability model, $\vg{\pcaa}$
based on the same ideas.

We start by defining the model $V_1(\pcaa)$. As for $V(\pcaa)$, we
define this using inductive definitions. $V(\pcaa)$ is the smallest
class $X$ satisfying
\[
\powset{2 \times |\pcaa| \times X} \subseteq X
\]

One can think of $V_1(\pcaa)$ as things from $V(\pcaa)$ with an
extra label from $2$. Hence given any element of $V_1(\pcaa)$ we can
think of it as an element of $V(\pcaa)$ by ignoring this extra
label. Explicitly, we define this recursively as follows. Given $a \in
V_1(\pcaa)$,
\[
a^\circ := \{\langle e, b^\circ \rangle \;|\; \langle s, e, b \rangle
\in a \}
\]

We write $e \vdone \phi$ to mean that $e \Vdash \phi$ in $V(\pcaa)$ when
each parameter, $a$, in $\phi$ has been replaced by $a^\circ $.

\begin{definition}
Let $\pcaa$ be a pca. We say that $\alpha$ is an \emph{automorphism}
of $\pcaa$ if it is a bijection $\pcaa \rightarrow \pcaa$ that such
that both $\alpha$ and $\alpha^{-1}$ preserve application and fix
$\pcas$ and $\pcak$.
\end{definition}

Given an automorphism, $\alpha$, of $\pcaa$, we can lift this
inductively to $V_1(\pcaa)$ as follows:

\[
\alpha(a) = \{ \langle 0, \alpha(e), \alpha(b) \rangle \; | \;
\langle 0, e, b \rangle \in a \} \cup \{ \langle 1, \alpha(e),
\alpha(b) \rangle \; | \; \langle 1, e, b \rangle \in a \}
\]

So this is simply the natural action of the automorphism group on
$V_1(\pcaa)$.

We assume that the pairing and projection elements and numerals
that appear in the definition of realizability over $V(\pcaa)$ are
defined using $\pcas$ and $\pcak$ and therefore fixed by any
automorphism. Hence we get
\begin{proposition}
\label{automspresrealiz}
Suppose that $\alpha$ is an automorphism of $\pcaa$ and 
\[
e \Vdash \phi
\]
Then if $\phi^\alpha$ is the result of replacing any parameters $c$ in
$\phi$ by $\alpha(c)$, we have
\[
\alpha(e) \Vdash \phi^\alpha
\]
\end{proposition}

Recall that normal filters are defined as follows.
\begin{definition}
Let $G$ be a group. Then a set of subgroups, $\Gamma$, is a
\emph{normal filter} on $G$ if

\begin{enumerate}
\item $G \in \Gamma$
\item $H \in \Gamma$ and $H$ is a subgroup of $H'$ implies that $H'
  \in \Gamma$
\item $H, H' \in \Gamma$ implies that $H \cap H' \in \Gamma$
\item $H \in \Gamma$ and $g \in G$ implies $g H g^{-1} \in \Gamma$
\end{enumerate}
\end{definition}

Recall that stabilisers are defined as follows.
\begin{definition}
If $G$ is a group of automorphisms of $\pcaa$ and $a \in V_1(\pcaa)$,
we define the \emph{stabiliser} of $a$ in $G$ as
\[
\stab{G}{a} := \{ \alpha \in G \;|\; \alpha(a) = a \}
\]
\end{definition}

\begin{definition}
Given a group $G$ of automorphisms of $\pcaa$ and $\Gamma$, a normal
filter on $G$, we define the class $\vg{\pcaa}
\subseteq V_1(\pcaa)$, of \emph{partly symmetric} sets inductively as
follows.

$\vg{\pcaa}$ is the smallest class $X$ such that
\[
\{ a \in V_1(\pcaa) \;|\; \stab{G}{a} \in \Gamma \text{ and }
(\forall \langle 0, e, b \rangle \in a)\, b \in X\} \subseteq X
\]
\end{definition}
In other words, $a \in V_1(\pcaa)$ is partly symmetric if it has a
``large'' stabiliser and every element that has been labelled with a
$0$ is also partly symmetric.  Note that this property is preserved by
automorphisms, and one can easily show the following.
\begin{proposition}
\label{sympreserved}
If $a \in V_1(\pcaa)$ and $\alpha$ is an automorphism of $\pcaa$ such
that $\alpha(a) = a$, then $\alpha(a^\circ ) = a^\circ$.
\end{proposition}
In particular if we take an element of $\vg{\pcaa}$, then it still
has $\stab{G}{a} \in \Gamma$ when we consider it as an element of
$V(\pcaa)$.

We can now define realizability on $\vg{\pcaa}$ as follows

\begin{eqnarray*}
e \vdzero a \in b & \text{iff} & \text{there is } \langle 0, (e)_0, c \rangle \in
b \text{ such that } (e)_1 \vdzero a = c \\
e \vdzero a = b & \text{iff} & \text{for every } \langle 0, f, c \rangle \in a,\,
(e)_0 f \vdzero c \in b \text{ and} \\
& & \qquad \text{for every } \langle 0, f, c \rangle
\in b,\, (e)_1 f  \vdzero c \in a, \text{ and } e \vdone a = b\\
e \vdzero \phi \wedge \psi & \text{iff} & (e)_0 \vdzero \phi \text{
  and }
(e)_1 \vdzero \psi \\
e \vdzero \phi \vee \psi & \text{iff} & \text{either } (e)_0 =
\underline{0} \text{ and }
(e)_1 \vdzero \phi, \text{ or } (e)_0 = \underline{1} \text{ and } (e)_1 \vdzero
\psi \\
e \vdzero \phi \rightarrow \psi & \text{iff} & f \vdzero \phi \text{
  implies } e f \vdzero \psi, \text{ and } e \vdone \phi \rightarrow
\psi \\
e \vdzero (\exists x \in a)\,\phi(x) & \text{iff} & \text{there is } \langle 0,
(e)_0, b \rangle \in a \text{ such that } (e)_1 \vdzero \phi(b) \\
e \vdzero (\forall x \in a)\,\phi(x) & \text{iff} & \text{for every } \langle 0,
f, b \rangle \in a,\, e f \vdzero \phi(b), \text{ and } e \vdone (\forall x
\in a)\,\phi(x) \\
e \vdzero (\exists x)\,\phi(x) & \text{iff} & \text{there is } a \in \vg{\pcaa}
\text{ such that } e \vdzero
\phi(a) \\
e \vdzero (\forall x)\,\phi(x) & \text{iff} & \text{for every } a \in \vg{\pcaa},\,
e \vdzero \phi(a), \text{ and } e \vdone (\forall x)\,\phi(x)
\\
e \vdzero \neg \phi & \text{iff} & f \vdone \phi \text{ is false for
  every } f \in \pcaa
\end{eqnarray*}

We write $\vg{\pcaa} \models \phi$ to mean that there is some $e \in
\pcaa$ such that $e \vdzero \phi$.

We clearly have the following proposition.

\begin{proposition}
Suppose that $\alpha$ is an automorphism and
$$
e \vdzero \phi
$$
Then, writing $\phi^\alpha$ for the formula obtained by replacing any
parameters, $a$, in $\phi$ by $\alpha(a)$,
$$
\alpha(e) \vdzero \phi^\alpha
$$
\end{proposition}

The definition above can be seen as a combination of realizability and
Kripke models of intuitionistic logic. See for example
\cite{troelstravandalen} for a description of Kripke models. Like in
\cite{FriedmanScedrov85}, the poset used in this model would have just
two elements.  On this basis, one should not be surprised by the
following proposition.

\begin{proposition}
\label{realpreserve}
Suppose that $e \vdzero \phi$. Then also $e \vdone \phi$.
\end{proposition}

\begin{proof}
We show this by induction on formulae, $\phi$. One can see that the
definition of $\vdzero$ has been carefully chosen so that we can
perform the induction at $=$, universal quantifiers, implication, and
negation. One can check that the induction holds at conjunction,
disjunction, $\in$, and existential quantifiers.
\end{proof}

\subsection{Soundness Theorems}
We now need to show soundness for intuitionistic logic and the axioms of
$\czf$.

Throughout the soundness theorems, the following proposition is useful.

\begin{proposition}
\label{multimpl}
  \begin{enumerate}
  \item To show $e \vdzero (\forall x_1)\ldots(\forall
    x_n)\,\phi(x_1,\ldots,x_n)$ it is sufficient to show that for all
    $a_1, \ldots, a_n \in \vg{\pcaa}$,
    \[ e \vdzero \phi(a_1,\ldots,a_n) \] and for all $a_1,\ldots,a_n
    \in V_1(\pcaa)$,
    \[ e \vdone \phi(a_1,\ldots,a_n) \]
  \item To show $e \vdzero \phi_1 \rightarrow (\phi_2 \rightarrow
    (\ldots \rightarrow (\phi_n \rightarrow \psi)\ldots)$, it is
    sufficient to show that for any $e_1,\ldots,e_{n - 1} \in \pcaa$,
    $e e_1 \ldots e_{n - 1} \downarrow$ and that whenever $e_i \vdzero
    \phi_i$ for each $i = 1,\ldots,n$ we have
    \[ e e_1 \ldots e_n \vdzero \psi \]
    and whenever $e_i \vdone \phi_i$ for each $i = 1,\ldots,n$ we have
    \[ e e_1 \ldots e_n \vdone \psi \]
  \end{enumerate}
\end{proposition}

\begin{proof}
Both parts can be proved by induction on $n$.
\end{proof}

\subsection{First Order Logic}

\begin{proposition}
$\vg{\pcaa}$ satisfies soundness for first order logic.
\end{proposition}
Explicitly this means that for every axiom $\phi$ of the intuitionistic
predicate calculus, $\vga \models \phi$, and for every inference rule,
$\frac{\phi_1, \ldots, \phi_n}{\psi}$, if $\vga \models \phi_i$ for $i =
1,\ldots,n$, then $\vga \models \psi$.

The reader may wish to compare the following with the
proofs for soundness of intuitionistic logic in realizability and
Kripke models in, for example, \cite{troelstravandalen}.

\subsubsection{Axioms}
The axioms of intuitionistic predicate calculus are as follows.

\begin{enumerate}
\item $\phi \rightarrow (\psi \rightarrow \phi)$
\item $(\phi \rightarrow (\psi \rightarrow \chi)) \rightarrow
((\phi \rightarrow \psi) \rightarrow (\phi \rightarrow \chi))$
\item $\phi \rightarrow (\psi \rightarrow \phi \wedge \psi)$
\item $\phi \wedge \psi \rightarrow \phi$
\item $\phi \wedge \psi \rightarrow \psi$
\item $\phi \rightarrow \phi \vee \psi$
\item $\psi \rightarrow \phi \vee \psi$
\item $(\phi \vee \psi) \rightarrow ((\phi \rightarrow \chi)
  \rightarrow ((\psi \rightarrow \chi) \rightarrow \chi))$
\item $(\phi \rightarrow \psi) \rightarrow ((\phi \rightarrow \neg \psi) 
\rightarrow \neg \phi)$
\item $\phi \rightarrow (\neg \phi \rightarrow \psi)$
\item $(\forall x)\phi(x)\; \rightarrow\; \phi(y)$, where $y$ is free for
  $x$ in $\phi(x)$
\item $\phi(y) \rightarrow (\exists x)\phi(x)$, where $y$ is free for $x$ in
  $\phi(x)$
\end{enumerate}

As an example we will prove 2 and 11. These demonstrate the main ideas
that are used for the remaining axioms.



\paragraph{2}
We claim that 
\[
\pcas \vdzero (\phi \rightarrow (\psi \rightarrow \chi)) \rightarrow
((\phi \rightarrow \psi) \rightarrow (\phi \rightarrow \chi))
\]
By proposition \ref{multimpl}, it is enough to show that whenever 
  $e  \vdzero  \phi \rightarrow (\psi \rightarrow \chi)$
  $f  \vdzero  \phi \rightarrow \psi$ and
  $g  \vdzero  \phi$
we have 
\[
\pcas e f g \vdzero \chi
\]
and whenever
  $e  \vdone  \phi \rightarrow (\psi \rightarrow \chi)$
  $f  \vdone  \phi \rightarrow \psi$ and
  $g  \vdone  \phi$
we have 
\[
\pcas e f g \vdone \chi
\]
However, $\pcas e f g = e g (f g)$, so one can easily check that this
is the case. (We also have that by definition $\pcas e f \downarrow$
for all $e$ and $f$.)






\paragraph{11}
Let $I := \pcas \pcak \pcak$ be the identity. Then we claim
\[
I \vdzero (\forall x)\phi(x) \;\rightarrow\; \phi(y)
\]
What we actually mean is the universal closure of
this axiom. Without loss of generality we can assume the universal
closure is (ignoring any additional parameters) the following:
\[
I \vdzero (\forall y)((\forall x)\phi(x) \;\rightarrow\; \phi(y))
\]
Expanding this out, this means that for $b \in \vg{\pcaa}$,
\[
I \vdzero (\forall x)\phi(x) \;\rightarrow\; \phi(b)
\]
and for $b \in V_1(\pcaa)$
\[
I \vdone (\forall x)\phi(x) \;\rightarrow\; \phi(b)
\]
So suppose that $e \vdzero (\forall x)\phi(x)$ and $b \in
\vg{\pcaa}$. Then in particular,
\[
e \vdzero \phi(b)
\]
If $e \vdone (\forall x)\phi(x)$, then
\[
e \vdone \phi(b)
\]
So we have shown
\[
I \vdzero (\forall x)\phi(x) \;\rightarrow\; \phi(b)
\]
However, we can similarly show that for $b \in V_1(\pcaa)$
\[
I \vdone (\forall x)\phi(x) \;\rightarrow\; \phi(b)
\]
So we can deduce 
\[
I \vdzero (\forall y)((\forall x)\phi(x) \;\rightarrow\; \phi(y))
\]
as required.


\subsubsection{Inference Rules}
The inference rules of IPL are
\begin{enumerate}
\item $\frac{\phi, \phi \rightarrow \psi}{\psi}$
\item $\frac{\psi \rightarrow \phi(x)}{\psi \rightarrow (\forall
  x)\phi(x)}$ where $x \notin FV(\psi)$
\item $\frac{\phi(x) \rightarrow \psi}{(\exists x)\phi(x) \rightarrow
  \psi}$ where $x \notin FV(\psi)$
\end{enumerate}

\paragraph{1 (Modus Ponens)}
Note first that we can assume that
\begin{eqnarray*}
e & \vdzero & (\forall x_1,\ldots,x_n)\,\phi(x_1,\ldots,x_n) \\
f & \vdzero & (\forall x_1,\ldots,x_n)\;\phi(x_1,\ldots,x_n) \rightarrow
\psi(x_1,\ldots,x_n) \\
\end{eqnarray*}
where the free variables for $\phi$ and $\psi$ are amongst $x_1,
\ldots, x_n$.

Then for any $a_1, \ldots, a_n \in \vg{\pcaa}$, 
\begin{eqnarray*}
e & \vdzero & \phi(a_1,\ldots,a_n) \\
f & \vdzero & \phi(a_1,\ldots,a_n) \rightarrow
\psi(a_1,\ldots,a_n) \\
\end{eqnarray*}
and hence
$
e f \vdzero \psi(a_1,\ldots,a_n)
$.
Similarly, for any $a_1, \ldots, a_n \in V_1(\pcaa)$,
$
e f \vdone \psi(a_1,\ldots,a_n)
$.
So we have shown
\[
e f \vdzero (\forall x_1,\ldots,x_n)\,\psi(x_1,\ldots,x_n)
\]

\paragraph{2}
We show that if 
$
e \vdzero \psi \rightarrow \phi(x)
$,
then
\[
e \vdzero \psi \rightarrow (\forall x)\,\phi(x)
\]
So suppose that
$
e \vdzero \psi \rightarrow \phi(x)
$.
Then, more explicitly (ignoring any additional free variables) this is
\[
e \vdzero (\forall x)\,(\psi \rightarrow \phi(x))
\]
In particular, if $a \in \vg{\pcaa}$, then
\[
e \vdzero \psi \rightarrow \phi(a)
\]
and if $a \in V(\pcaa)$ then
\[
e \vdone \psi \rightarrow \phi(a)
\]
Now suppose that 
$
f \vdzero \psi
$.
We need to show that for any $a \in \vg{\pcaa}$,
$
e f \vdzero \phi(a)
$
and for any $a \in V_1(\pcaa)$,
$
e f \vdone \phi(a)
$.
But this is clear from the above, so we can deduce
\[
e f \vdzero (\forall x)\,\phi(x)
\]
We can similarly show that if $f \vdone \psi$, then
\[
e f \vdone (\forall x)\,\phi(x)
\]
and so
\[
e \vdzero \psi \rightarrow (\forall x)\,\phi(x)
\]

\paragraph{3}
We claim that if
$
e \vdzero \phi(x) \rightarrow \psi
$,
then
$
e \vdzero (\exists x)\phi(x) \rightarrow \psi
$.
First note as before, that what we actually assume is that 
\[
e \vdzero (\forall x)\,(\phi(x) \rightarrow \psi)
\]
Now suppose that
$
f \vdzero (\exists x)\,\phi(x)
$.
Then there is $a \in \vg{\pcaa}$ such that
$
f \vdzero \phi(a)
$.
But we know from the above that
\[
e \vdzero \phi(a) \rightarrow \psi
\]
And so,
\[
e f \vdzero \psi
\]
We similarly know that if $f \vdone (\exists x)\,\phi(x)$, then $e f
\vdone \psi$. So we can deduce
\[
e \vdzero (\exists x)\,\phi(x) \;\rightarrow\; \psi
\]

\subsubsection{Axioms of Equality}
\begin{proposition}
\label{vgidentsound}
One can construct realizers $\mathbf{i}_r, \mathbf{i}_s, \mathbf{i}_t
,\mathbf{i}_0, \mathbf{i}_1$ such that
\begin{enumerate}
\item $\mathbf{i}_r \vdzero (\forall x)\,x = x$
\item $\mathbf{i}_s \vdzero (\forall x, y)\;x = y \rightarrow y = x$
\item $\mathbf{i}_t \vdzero (\forall x, y, z)\,(x = y \rightarrow (y = z
  \rightarrow x = z))$
\item $\mathbf{i}_0 \vdzero (\forall x, y, z)\,(x = y \rightarrow (y \in
  z  \rightarrow x \in z))$
\item $\mathbf{i}_1 \vdzero (\forall x, y, z)\,(x = y \rightarrow (z \in
  x \rightarrow z \in y))$
\end{enumerate}

Furthermore, for each formula (without parameters), $\phi(x,
z_1,\ldots,z_n)$, there is $\mathbf{i}_\phi$ such that
\[
\mathbf{i}_\phi \vdzero x = y \rightarrow (\phi(x, z_1,\ldots,z_n)
\rightarrow \phi(y, z_1,\ldots,z_n))
\]
\end{proposition}

\begin{proof}
We take these realizers from the proof of theorem 6.3 in
\cite{mccarty} and check that they still work in this context.

Define $\mathbf{i}_r$ from the fixed point theorem so that
\begin{eqnarray*}
((\mathbf{i}_r)_0 f)_0 & = & f \\
((\mathbf{i}_r)_0 f)_1 & = & \mathbf{i}_r \\
((\mathbf{i}_r)_1 f)_0 & = & f \\
((\mathbf{i}_r)_1 f)_1 & = & \mathbf{i}_r 
\end{eqnarray*}

In order to show $\mathbf{i}_r \vdzero (\forall x)x = x$, by
proposition \ref{multimpl}, what we need to show is
\begin{enumerate}
\item for every $a \in \vg{\pcaa}$, $\mathbf{i}_r \vdzero a = a$
\item for every $a \in V_1(\pcaa)$, $\mathbf{i}_r \vdone a = a$
\end{enumerate}

However note that the second of these conditions is basically the same
as the statement $\mathbf{i}_r \Vdash a = a$ in $V(\pcaa)$. Hence we
only have to check the first condition.

Furthermore, since we already know that for every $a \in \vg{\pcaa}$,
$\mathbf{i}_r \vdone a = a$, all we have to check is the following:

\[
\forall \langle 0, f, b \rangle \in a \; (\mathbf{i}_r)_0 \vdzero b
\in a
\]
and
\[
\forall \langle 0, f, b \rangle \in a \; (\mathbf{i}_r)_1 \vdzero b
\in a
\]

We show by induction that these conditions hold for every $a \in
\vg{\pcaa}$.

Suppose that $\langle 0, f, b \rangle \in a$. Then since this has been
labelled with $0$, we know that $b$ \emph{is also partly
  symmetric}. Also $b$ appears earlier in the inductive definition of
$V_1(\pcaa)$, so we can apply induction here and the above arguments
to get

\[
\mathbf{i}_r \vdzero b = b
\]

However, recall that we defined $\mathbf{i}_r$ using the fixed point
theorem so that for all $f$,

\begin{eqnarray*}
((\mathbf{i}_r)_0 f)_0 & = & f \\
((\mathbf{i}_r)_0 f)_1 & = & \mathbf{i}_r \\
\end{eqnarray*}
(and the same equations for $(\mathbf{i}_r)_1$).

Hence $\mathbf{i}_r \vdzero a = a$ as required.

The proof that $\mathbf{i}_s$ works as required is trivial and still
holds here.

$\mathbf{i}_t$, $\mathbf{i}_0$ and $\mathbf{i}_1$ are also the same as
in \cite{mccarty} and the proofs that they are as required can be
similarly adapted to this context.

The $\mathbf{i}_\phi$ are constructed by induction on the construction
of $\phi$. We will explicitly show how to do this for unbounded
universal quantifiers and implication since these contain the main
ideas for the rest of the induction.

We first show how to construct $\mathbf{i}_{\phi \rightarrow \psi}$.

Suppose that $a, b, c \in \vg{\pcaa}$, $e \vdzero a = b$ and $f \vdzero
\phi(a, c) \rightarrow \psi(a, c)$. Suppose further that
\[
g \vdzero \phi(b, c)
\]

Then 
\[
\mathbf{i}_\phi (\mathbf{i}_s e) g \vdzero \phi(a, c)
\]
and so
\[
f (\mathbf{i}_\phi (\mathbf{i}_s e) g) \vdzero \psi(a, c)
\]
and finally
\[
\mathbf{i}_\psi e (f (\mathbf{i}_\phi (\mathbf{i}_s e) g))
\vdzero \psi(b, c)
\]

Hence we can apply similar reasoning for $\vdone$ and for $a, b, c \in
V_1(\pcaa)$ and use proposition
\ref{multimpl} to show that we can take $\mathbf{i}_{\phi \rightarrow
  \psi}$ to be
\[
\mathbf{i}_{\phi \rightarrow \psi} := 
\lambda (x, y, z).\mathbf{i}_\psi x (y (\mathbf{i}_\phi (\mathbf{i}_s x) z))
\]

For unbounded universal quantifiers, we show that we can take
$\mathbf{i}_{(\forall z)\phi(x,z)} := \mathbf{i}_{\phi(x, z)}$.
Suppose that
\[
\mathbf{i}_{\phi(x, z)}  \vdzero  (\forall z)\,(x = y \rightarrow (\phi(x, z)
\rightarrow \phi(y, z)))
\]
and suppose that for $a,b \in \vg{\pcaa}$, $e \vdzero a = b$ and
\[
f \vdzero (\forall z)\,\phi(a, z)
\]
Then for all $c \in \vg{\pcaa}$,
\[
f \vdzero \phi(a, c)
\]
and so
\[
\mathbf{i}_{\phi(x, z)} e f \vdzero \phi(b, c)
\]
One can check the corresponding case for $c \in V_1(\pcaa)$ to get
\[
\mathbf{i}_{\phi(x, z)} e f \vdzero (\forall z)\,\phi(b, z)
\]
as required.
\end{proof}

\begin{proposition}
  We can find realizers for the axioms for bounded quantifiers. That
  is, we can find realizers for the following statements.

\begin{enumerate}
\item $(\forall x \in a)\,\phi(x) \;\rightarrow\; (\forall x)\,(x \in a
  \rightarrow \phi(x))$
\item $(\forall x)\,(x \in a \rightarrow \phi(x)) \;\rightarrow\; (\forall x
  \in a)\,\phi(x)$
\item $(\exists x \in a)\,\phi \;\rightarrow\; (\exists x)\,(x \in a \wedge
  \phi(x))$
\item $(\exists x)\,(x \in a \wedge \phi(x)) \;\rightarrow\; (\exists x \in
  a)\,\phi(x)$
\end{enumerate}
\end{proposition}

\begin{proof}
  The proof of theorem 4.3 from \cite{rathjen06} can easily be adapted
  using proposition \ref{vgidentsound} and applying proposition
  \ref{multimpl} where necessary.
\end{proof}

The following help illustrate the relation between
realizability in $\vg{\pcaa}$ and $V(\pcaa)$.
\begin{definition}
We say that $a \in \vg{\pcaa}$ is \emph{(completely) symmetric} if
every element of $a$ is of the form
\[
\langle 0, e, b \rangle
\]
where $b$ is completely symmetric. (This is an inductive definition).
\end{definition}

\begin{proposition}
\label{boundedpreserve1}
Suppose that $\phi$ is a bounded formula, all of whose parameters are
completely symmetric. Then
\[
e \vdzero \phi \text{ iff } e \vdone \phi
\]
\end{proposition}

\begin{proof}
When all parameters are completely symmetric the two definitions of
realizability agree for everything except unbounded quantifiers.
\end{proof}

We now move on to the proof of soundness for the axioms of set theory.
To make things easier, we assume a background universe of $\zfc$, and
show the soundness of $\izf$.

\begin{theorem}
$\vg{\pcaa}$ satisfies the axioms of $\izf$.
\end{theorem}
We first deal with what are sometimes referred to as ``set existence
axioms.'' That is, axioms of the form
\[
(\forall z_1,\ldots,z_n)(\forall x)(\exists y)\,\phi(x, y, z_1, \ldots, z_n)
\]
where the free variables of $\phi$ are amongst $x, y, z_1,\ldots,z_n$.
For these axioms we can apply proposition \ref{multimpl} to show that
it is sufficient to find $e$ such that for every $a, c_1, \ldots, c_n \in
V_1(\pcaa)$, there is $b \in V_1(\pcaa)$ such that 
\[
e \vdone \phi(a, b, c_1, \ldots, c_n)
\]
and for every $a, c_1, \ldots, c_n \in \vg{\pcaa}$ there is $b \in
\vg{\pcaa}$ such that 
\[
e \vdzero \phi(a, b, c_1, \ldots, c_n)
\]

However, the first of these statements follows from the soundness
theorem for $V(\pcaa)$. Hence we only have to check the second of
these conditions.

\paragraph{Separation}
By the above reasoning, what we need to show is the following
statement:

Suppose that $e$ is the usual realizer for separation from
\cite{mccarty} or \cite{rathjen06}, $A$ is a partly symmetric set, and
$\phi(x)$ is a formula with partly symmetric parameters. Then there is
a partly symmetric set, $S$, such that

\[
e \vdzero ((\forall x \in A)\;\phi(x) \rightarrow x \in S) \;\wedge\; 
((\forall x \in S)\;x \in A \wedge \phi(x))
\]

We construct this $S$ as follows:
\begin{eqnarray*}
S_0 & = & \{ \langle 0, \mathbf{p}f g, a \rangle \; | \; \langle 0, f, a
\rangle \in A \,\wedge\, g \vdzero \phi(a) \} \\
S_1 & = & \{ \langle 1, \mathbf{p}f g, a \rangle \; |\; \langle s, f, a
\rangle \in A \,\wedge\, g \vdone \phi(a) \} \\
S & = & S_0 \cup S_1 \\
\end{eqnarray*}

Suppose that $H$ is the intersection of stabilisers of $A$ and all the
parameters of $\phi$. Note that $H \in \Gamma$.

Let $\alpha \in H$, and $\langle 0, \mathbf{p}f g, a \rangle \in
S_0$. Then $\langle 0, f, a \rangle \in A$ and $g \vdzero
\phi(a)$. Since $\alpha \in H$, we know that $\langle 0, \alpha(f),
\alpha(a) \rangle \in A$ and $\alpha(g) \vdzero
\phi(\alpha(a))$. Hence we also have $\langle 0, \alpha(\mathbf{p}f
g), \alpha(a) \rangle \in S_0$. One can show the same result for
$S_1$ and hence get $\stab{G}{S} \in \Gamma$. Note further that if 
$\langle 0, \mathbf{p}f g, a \rangle \in S$ then also $\langle 0, f, a
\rangle \in A$ and so $a$ is partly symmetric. We can now deduce that
$S$ is partly symmetric.

One can easily check that the usual realizer does still work for $S$.

\paragraph{Power Set}

As before, note that we only have to check power set for partly
symmetric sets. Hence let $A \in \vg{\pcaa}$.

Let
\begin{eqnarray*}
P_0 & = & \{ \langle 0, e, b \rangle \;|\; b \in \vg{\pcaa}, e \vdzero
b \subseteq A \} \\ P_1 & = & \{ \langle 1, e, b \rangle \;|\; b \in
V_1(\pcaa), e
\vdone b \subseteq A \} \\ P & = & P_0 \cup P_1 
\end{eqnarray*}

To show that $P_0$ and $P_1$ are both sets, one can use a notion of
rank as in chapter 2 of \cite{mccarty}.

Alternatively, we can prove that $P_0$ and $P_1$ are sets as follows.
Show by induction (and power set in the background universe) that for
any $a \in V_1(\pcaa)$, $\{b \in V_1(\pcaa) \;|\; (\exists e \in
\pcaa) \, e \vdone b = a \}$ is a set and deduce that for $a \in
\vg{\pcaa}$, $\{b \in
\vg{\pcaa} \;|\; (\exists e \in \pcaa)\; e \vdzero b = a \}$ is also a
set. Then this implies that for any $a \in V_1(\pcaa)$, $\{ b \in
V_1(\pcaa) \;|\; (\exists e \in \pcaa)\, e \vdone b \subseteq a \}$ is
a set and if $a \in \vg{\pcaa}$, $\{ b \in \vg{\pcaa} \;|\; (\exists e
\in \pcaa)\, e \vdzero b \subseteq a \}$ is a set, and hence $P_0$,
$P_1$ and $P$ are also sets.

Now note that if $e \vdzero b \subseteq A$ and $\alpha \in
\stab{G}{A}$ then $\alpha(e) \vdzero \alpha(b) \subseteq A$, and
similarly if $e \vdone b \subseteq A$, and we have ensured that any
elements of $P$ labelled with $0$ are partly symmetric. Hence $P$ is
partly symmetric.

One can easily show that the realizer in \cite{mccarty} still works
here.

\paragraph{Union}
We assume that we are given a set $A \in \vg{\pcaa}$ and
construct a set to show the union axiom. Since we already have full
separation, we only have to construct a $U$ such that we have a
realizer for $(\forall x \in A)(\forall y \in x)\,y \in U$.

Let
\[
U = \{\langle 0, \underline{0}, b \rangle \; |\; \langle 0, e, c \rangle
\in A, \langle 0, f, b \rangle \in c \} \cup
\{ \langle 1, \underline{0}, b \rangle \;|\; \langle s, e, c \rangle \in
A, \langle s', f, b \rangle \in c \}
\]
Note that
\[
(\pcak (\pcak (\mathbf{p} \underline{0} \mathbf{i}_r)))
\vdzero (\forall x \in A)(\forall y \in x)\,y \in U
\]

\paragraph{Pair}
Given $a, b \in \vg{\pcaa}$, consider the set
\[
P = \{ \langle 0, \underline{0}, a \rangle, \langle 0, \underline{1}, b \rangle \}
\]
We can easily see that
\[
e \vdzero (\forall x)(x \in P \leftrightarrow (x = a \vee x = b))
\]

\paragraph{Infinity}
We check that the proof in \cite{rathjen06} still holds here. We use
the same $\bar{\omega}$ as in section \ref{vaint}. We write $\bot_v$
for the formula $(\forall x \in v)\;\bot$, and write $SC(x, y)$ for $y =
x \cup \{x\}$ (expressed as a bounded formula).

Note first that we can apply proposition \ref{boundedpreserve1} and
the soundness theorem in \cite{rathjen06} to reduce the problem to
finding a realizer for
\[
(\forall v)((\bot_v \vee (\exists u \in \bar{\omega})\, SC(u, v))
\rightarrow v \in \bar{\omega})
\]

Since we can clearly find a realizer to show that the empty set is in
$\bar{\omega}$, this is reduced to finding a realizer for
\[
(\forall v)((\exists u \in \bar{\omega})\,SC(u, v)
\;\rightarrow\; v \in \bar{\omega})
\]

Hence we assume that there is $a \in \vg{\pcaa}$ with
$
e \vdzero (\exists u \in \bar{\omega})\;SC(u, a)
$.
So there must be some $n$ such that $(e)_0 = \underline{n}$ and
$
(e)_1 \vdzero SC(\overline{n}, a)
$.

One can clearly find a realizer for $ SC(\overline{n}, \overline{n +
  1}) $ and hence a realizer, using the soundness of extensionality
(once we have checked this) for $ SC(u, v) \wedge SC(u, v')
\rightarrow v = v' $.  We can use these to construct a realizer for $
a \in \overline{\omega} $, as required.

\paragraph{Collection}
Assume 
\[
e \vdzero (\forall x \in A)(\exists y)\,\phi(x, y)
\]
where $\phi$ is a formula with all parameters partly symmetric.

By collection in the background universe, we can find a $C_0$ such
that whenever $\langle 0, f, a \rangle \in A$, there is $\langle 0,
\underline{0}, c \rangle \in C_0$ such that $c$ is partly symmetric and
$e.f \vdzero \phi(a, c)$. Note that $C'_0 := C_0 \cap \{0\} \times
\{\underline{0}\} \times \vg{\pcaa}$ still has this property, but is an
element of $V_1(\pcaa)$ such that for every $\langle 0, g, c \rangle
\in C'_0$, $c$ is partly symmetric.

Similarly, there is a $C_1'$, such that every element of $C'_1$ is of
the form $\langle 1, \underline{0}, c \rangle$ with $c \in V_1(\pcaa)$
and whenever $\langle s, f, a \rangle \in A$, there is $\langle 1,
\underline{0}, c \rangle \in C_1'$ such that $e f \vdone \phi(a, c)$.

Let $C = C_0' \cup C_1'$, and let $C'$ be the closure of $C$ under all
automorphisms in $G$. Note that $C' \in \vg{\pcaa}$ and this set
together with the usual realizer from \cite{mccarty} is enough to show 
the soundness of collection.

\paragraph{Extensionality}
One can check that the realizers for the formula
\[
((\forall x \in a)\,x \in b) \wedge ((\forall x \in b)\,x \in a)
\]
in fact are already realizers for $a = b$, so we can use the identity
to show extensionality (in this form).

\paragraph{$\in$-Induction}

Suppose that $$e \vdzero (\forall y)((\forall x \in y)\,\phi(x)
\;\rightarrow\; \phi(y))$$

Let $e' = \lambda (x, y).e x$ and let $f$ be given by the fixed point
theorem so that for all $g$
\[
f g \simeq e' f g
\]
Note that we know
\[
e \vdone (\forall y)((\forall x \in y)\,\phi(x)
\;\rightarrow\; \phi(y))
\]
and so by the usual proof we have that for all $a \in V_1(\pcaa)$, and
all $g \in \pcaa$, $f g \vdone \phi(a)$. We claim that for all $a \in
\vg{\pcaa}$, and all $g \in \pcaa$, $f g \downarrow$ and $f g \vdzero
\phi(a)$.

So suppose that $a \in \vg{\pcaa}$. Then for every $\langle 0, g, b
\rangle \in a$, we know by induction in the background universe (since
$b$ must be partly symmetric and appears earlier in the inductive
definition of $V_1(\pcaa)$ than $a$) that $f g \downarrow$ and $f g
\vdzero \phi(b)$. We also know from the above that $f \vdone (\forall
x \in a)\,\phi(x)$. Hence $f \vdzero (\forall x \in a)\,\phi(x)$. Thus
we have for any $g \in \pcaa$, $e' f g \simeq e f$ (is defined and)
realizes $\phi(a)$. But $e' f g \simeq f g$ and so $f g \vdzero
\phi(a)$ as required.

\qed

\begin{remark}
  Note that when we proved the axiom of infinity we used the same
  standard representation $\overline{\omega}$ as for $V(\pcaa)$. Note
  further that if $f \in \pcaa$ is such that for all $n \in \omega$
  there is $m \in \omega$ with $f \underline{n} = \underline{m}$,
  then the $\overline{f}$ from section \ref{vaint} is completely
  symmetric and hence we have the same standard representations of the
  naturals and Baire space as we did before.
\end{remark}

\section{The Model $\vip{\pcaa}$}
We say that $a \in V(\pcaa)$ is \emph{functional} if for any
$\langle e, b \rangle, \langle e', b' \rangle \in a$, if $e = e'$ then
$b = b'$.

We will define $\vip{\pcaa}$ using inductive definitions. Define the
operator $\powip$ by
\[
\powip(X) := \{ a \in \powset{|\pcaa| \times X} \;|\; a \text{ is
  functional}\}
\]
We then define $\vip{\pcaa}$ as the smallest class $X$ satisfying 
\[
\powip(X) \subseteq X
\]

We define realizability on $\vip{\pcaa}$ as follows. We write $\vdone$
for realizability at $V(\pcaa)$.

\begin{eqnarray*}
e \vdzero a \in b & \text{iff} & \text{there is } \langle (e)_0, c \rangle \in
b \text{ such that } (e)_1 \vdzero a = c \\
e \vdzero a = b & \text{iff} & \text{for every } \langle f, c \rangle \in a,\,
(e)_0 f \vdzero c \in b \text{ and } \\
& & \qquad \text{for every } \langle f, c \rangle
\in b,\, (e)_1 f \vdzero c \in a \\
e \vdzero \phi \wedge \psi & \text{iff} & (e)_0 \vdzero \phi \text{
  and }
(e)_1 \vdzero \psi \\
e \vdzero \phi \vee \psi & \text{iff} & \text{either } (e)_0 =
\underline{0} \text{ and }
(e)_1 \vdzero \phi, \text{ or } (e)_0 = \underline{1} \text{ and } (e)_1 \vdzero
\psi \\
e \vdzero \phi \rightarrow \psi & \text{iff} & f \vdzero \phi \text{
  implies } e f \vdzero \psi \text{ and } e \vdone \phi \rightarrow
\psi \\
e \vdzero (\exists x \in a)\,\phi(x) & \text{iff} & \text{there is } \langle 
(e)_0, b \rangle \in a \text{ such that } (e)_1 \vdzero \phi(b) \\
e \vdzero (\forall x \in a)\,\phi(x) & \text{iff} & \text{for every } \langle 
f, b \rangle \in a,\, e f \vdzero \phi(b), \text{ and } e \vdone
(\forall x \in a)\,\phi(x)\\
e \vdzero (\exists x)\,\phi(x) & \text{iff} & \text{there is } a \in \vip{\pcaa} 
\text{ such that } e \vdzero
\phi(a) \\
e \vdzero (\forall x)\,\phi(x) & \text{iff} & \text{for every } a \in \vip{\pcaa},\,
e \vdzero \phi(a), \text{ and } e \vdone (\forall x)\,\phi(x) \\
e \vdzero \neg \phi & \text{iff} & f \vdone \phi \text{ is false for
  every } f \in \pcaa \\
\end{eqnarray*}
We write $\vip{\pcaa} \models \phi$ to mean that there is some $e \in
\pcaa$ such that $e \vdzero \phi$.

\begin{remark}
\label{boundedsame}
This is a much simpler embedding than that of
$\vg{\pcaa}$. We have not needed to alter the definition of
realizability for bounded universal quantification and equality in
order to ensure realizability is preserved. Hence, realizability for
bounded formulae is identical in $\vip{\pcaa}$ and $V(\pcaa)$. 
\end{remark}

\begin{proposition}
  $\vip{\pcaa}$ is sound with respect to the intuitionistic predicate
  calculus and satisfies the axioms of equality and the axioms for
  bounded quantifiers.
\end{proposition}

\begin{proof}
This follows by exactly the same proof as for $\vg{\pcaa}$.
\end{proof}

It remains to check that when we show the soundness of the axioms of
$\czf$, we can assume the sets we construct are functional.
Since we will require choice in the background universe for this
proof, we work over a background universe of $\zfc$.

\begin{theorem}
$\vip{\pcaa}$ is sound with respect to the axioms of $\czf$.
\end{theorem}

\paragraph{Extensionality}
This is the same as for $V(\pcaa)$.

\paragraph{Bounded Separation}
Given $A \in \vip{\pcaa}$ and a bounded formula, $\phi$, consider
the set
\[
S = \{ \langle \mathbf{p}e f, a \rangle \;|\; \langle e, a \rangle \in
A, f \vdzero \phi(a) \}
\]

Note that this is functional, since $A$ is, and since
realizability for bounded formulae is identical in $\vip{\pcaa}$ and
$V(\pcaa)$, we can see that this can be used to show the soundness of
bounded separation.

\paragraph{Pair}
Given $a, b \in V(\pcaa)$, consider
\[
P = \{ \langle \underline{0}, a \rangle, \langle \underline{1}, b \rangle \}
\]
This is clearly functional, and we can
easily use this to show the soundness of pair.

\paragraph{Strong Collection}
Suppose that 
\[
e \vdzero (\forall x \in A)(\exists y)\,\phi(x, y)
\]
For each $\langle f, a \rangle \in A$, we can assume by choice in the
background universe that we have chosen a $c_f \in \vip{\pcaa}$ such
that $e f \vdzero \phi(a, c_f)$ (and hence also $e f \vdone \phi(a,
c_f)$).

Let
\[
C = \{ \langle f, c_f \rangle \;|\; \langle f, a \rangle \in A \}
\]
This is clearly functional (since $A$ is).

Note that 
\[
\lambda x.\mathbf{p}  x (e x) \vdzero (\forall
x \in A)(\exists y \in C)\,\phi(x, y)
\]
and in fact we can use exactly the same realizer again in
\[
\lambda x.\mathbf{p}  x (e x) \vdzero (\forall
y \in C)(\exists x \in A)\,\phi(x, y)
\]
(since every element of $C$ is of the form $\langle f, c_f \rangle$
where $\langle f, x \rangle \in A$ and $e.f \vdzero \phi(x, c_f)$). So
we get soundness for strong collection.

\paragraph{Subset Collection}
Suppose we are given sets $A, B \in \vip{\pcaa}$. Suppose further that
$e \in \pcaa$ is such that for all $\langle f, a \rangle \in A$, $e f
\downarrow$ and there is $\langle e f, b \rangle \in B$ for some
$b$. In this case we can define 
\[
\overline{e} := \{ \langle f, b \rangle \; | \; \exists a \langle f, a
\rangle \in A, \langle e f, b \rangle \in B \}
\]
(Clearly $\overline{e} \in \vip{\pcaa}$).

Now let
\[
D := \{ \langle e, \overline{e} \rangle \;|\; e \in \pcaa,
\overline{e} \text{ is defined} \}
\]

Clearly $D \in \vip{\pcaa}$. We shall show that we can use $D$ to show
the soundness of subset collection.

Suppose that $u \in V(\pcaa)$ is such that
\[
e \vdone (\forall x \in A)(\exists y \in B)\,\phi(x, y, u)
\]
Let
\[
e' := \lambda x.(e x)_0
\]
Note that for every $\langle f, a \rangle \in A$, we have $e' f
\downarrow$ and there is (a unique) $b$ with $\langle e' f, b \rangle
\in B$, and so $\langle e', \overline{e'} \rangle \in D$. Furthermore
$(e f)_1 \vdone \phi(a, b, u)$, and so we can find a realizer for 
\[
(\forall x \in A)(\exists y \in \overline{e'})\,\phi(x, y, u) \;\wedge\;
(\forall y \in \overline{e'})(\exists x \in A)\,\phi(x, y, u)
\]
We can do exactly same if 
\[
e \vdzero (\forall x \in A)(\exists y \in B)\,\phi(x, y, u)
\]

Hence this does give a proof of the soundness of subset collection.

\paragraph{Union}
Suppose we have been given $A \in \vip{\pcaa}$. We want to find
an functional set that we can use to show the union axiom.
So let
\[
U = \{ \langle \mathbf{p} e f, c \rangle \;|\; \langle f, b \rangle
\in a, \langle e, c \rangle \in b \}
\]
Then we see that
\[
\lambda (x, y).(\mathbf{p} (\mathbf{p} y x) \mathbf{i}_r)
\vdzero (\forall x \in a)(\forall y \in x)\,y \in U
\]

\paragraph{Infinity}
We note that the $\overline{\omega}$ given in section \ref{vaint} is
functional, and since no other sets need to be constructed
in the proof of infinity, this means we can use the same proof as
usual here (see eg \cite{rathjen06}).

\paragraph{$\in$-Induction}
The same proof as for $\vg{\pcaa}$ still holds here.
\qed

\section{The Pca $\term$}
\label{sectionterm}

\subsection{Definition}
We will define a term model based on combinatory logic. This is
similar to the model $NT$ that appears in chapter 6 of
\cite{beeson85}.

We start by adding constants $\xi_i$ and $\zeta_F$ to the language of
combinatory logic.

\begin{definition}
\label{termdef}
The set, $\mathcal{C}$ of \emph{terms} is defined inductively as follows
\begin{enumerate}
\item constants $\pcas$ and $\pcak$ are terms
\item free variables $x_i$ for each $i \in \omega$ are terms
\item for each $i>0$, the constant $\xi_i$ is a term (we will call
  these \emph{atoms})
\item for each bijection $F:\omega_{>0} \rightarrow \omega_{>0}$ such
  that $F$ is the identity everywhere except for some finite set, the
  constant $\zeta_F$ is a term
\item if $s$ and $t$ are terms, then the ordered pair $\langle s, t
  \rangle$, written as $(s.t)$ or just $st$ is also a term
\end{enumerate}
\end{definition}

\begin{remark}
  Instead of using those $F$ that are the identity everywhere except
  for a finite set, the proofs in this paper will still hold using
  any subgroup of permutations of $\omega_{>0}$ that contains all
  transpositions. For example, we could alternatively let $F$ be any
  permutation of $\omega_{>0}$ or consider just the computable
  permutations.
\end{remark}

We will consider $\mathcal{C}$ as a term rewriting system. We have in
particular the two standard reduction rules

\begin{eqnarray*}
\pcas x y z & \rightarrow & x z (y z)  \\
\pcak x y & \rightarrow & x
\end{eqnarray*}

In addition to these, we add a new reduction rule. In the below let
$\underline{n}$ be $n$ encoded using $\pcas$ and $\pcak$ in the usual
way. Then we define the \emph{$\zeta$-rule}, or
\emph{$\zeta$-reduction} as follows:
\[
\zeta_F t  \rightarrow  \underline{n}
\]
where $t$ is a closed term and $n$ is either maximal such that $n =
F(m)$ where $\xi_m$ occurs in $t$ or $n = 0$ and no $\xi_m$ occurs in
$t$.

Note that this term rewriting system is ambiguous. That is, there are
terms that can be reduced in two incompatible ways. For example, the
term $\zeta_{\lambda x.x} (\pcak \pcak \xi_1)$ can reduce either to
$\underline{1}$ or to $\underline{0}$ depending on whether the subterm
$\pcak \pcak \xi_1$ is reduced before or after
$\zeta$-reduction. However, we still have a notion of normal form
(when no reduction rule can be applied to a term) and leftmost
innermost reduction, as defined below.

\begin{definition}
We define a sequence of partial operators, $\operatorname{RED}_n$ for
each $n$ as follows:

For $n = 0$, define $\operatorname{RED}_0$ as follows:

\begin{enumerate}
\item if $t$ is a normal form, $\lmimn{0}{t} = t$
\item for $t = \pcak r s$ where $r$ and $s$ are normal forms, 
  $\lmimn{0}{\pcak r s} = r$
\item for $t = \zeta_F r$ where $r$ is a normal form,
  $\lmimn{0}{\zeta_F r} = \underline{n}$ where $n$ is maximal such
  that $\xi_{F^{-1}(n)}$ occurs in $r$ or $0$ if no $\xi_i$ occurs in
  $r$
\end{enumerate}

If $\operatorname{RED}_{n}$ has been already been defined, then we
define $\operatorname{RED}_{n+1}$ as follows:

\begin{enumerate}
\item if $\lmimn{n}{t}\downarrow$, then $\lmimn{n+1}{t} =
  \lmimn{n}{t}$
\item for $t = \pcas r s u$, where $r$, $s$, and $u$ are normal forms,
  $\lmimn{n+1}{\pcas r s u} \simeq \lmimn{n}{\lmimn{n}{r u}
  \lmimn{n}{s u}}$
\item if $t = r s$ and neither of previous cases apply, then
  $\lmimn{n+1}{r s} \simeq \lmimn{n}{\lmimn{n}{r} \lmimn{n}{s}}$
\end{enumerate}

We then define $\operatorname{RED}$ as 
\[
\operatorname{RED} = \bigcup_{n \in \omega}\operatorname{RED}_n
\]

\end{definition}

Note that if $\lmim{t}$ is defined, then it is a normal form.

We now define our pca, $\term$

\begin{definition}
Let $\term$ be the set of closed normal forms of $\mathcal{C}$ together with
the following application:
\[
s.t := \lmim{s.t}
\]
(undefined if $\lmim{s.t}$ is undefined)
\end{definition}

Note that since this is a pca we can consider the notion of terms over
$\term$ (ie definition \ref{termoveradef}) as well as terms in the
sense of definition \ref{termdef}. Fortunately we are free to switch
between thinking of terms as elements of $\mathcal{C}$ and as terms
over $\term$ by the following proposition.  (Note that this
proposition is a characteristic of inside first reduction and is not
shared by some similar structures: see Remark 6.1.4 in
\cite{beeson85}.)

\begin{proposition}
\label{termdefsequiv}
Suppose that $t$ is a closed term over $\term$ (in the sense of
definition \ref{termoveradef}) and write $t^\ast $ for the
corresponding term (in the sense of definition \ref{termdef}). Then
$\lmim{t^\ast}$ is defined if and only if $t$ denotes, and in this
case we have
\[
\lmim{t^\ast } = t
\]
\end{proposition}

\begin{proof}
  This essentially appears as parts (i) and (ii) of lemma 6.1.1 in
  chapter 6 of \cite{beeson85}. We simply note that the proof still
  holds in this setting where we also have $\zeta$-reduction.
\end{proof}

\begin{proposition}
$\term$ is a pca.
\end{proposition}

\begin{proof}
Note firstly that $\pcas$ and $\pcak$ are normal terms and hence
elements of $\term$.

If $r$ and $s$ are normal forms, then so are $\pcak r$ and $\pcas r
s$. Hence $\pcak r \downarrow$, $\pcas r \downarrow$, and $\pcas r s
\downarrow$. Also $\lmim{\pcak r s} = r$, so $\pcak r s = r$. 

It remains only to check that for all $r, s, t$, $\pcas r s t \simeq r
t(st)$. However this is clear from the definition. (In fact the left
hand side is defined at stage $n + 1$ if and only if the right hand
side is defined at stage $n$.)
\end{proof}

\subsection{Preservation of Atoms}
\label{presatomssect}
The non trivial structure of $\vg{\term}$ will rely on the $\xi_i$,
and the rich supply of automorphisms arising from permutations of
them. We will want to ensure therefore that under suitable conditions
the atoms aren't eliminated by the realizability structure. In this
section, we will aim towards a lemma that will enable us to show
this.

\begin{definition}
For any pca, $\pcaa$, one may consider the following classes of
elements

\begin{enumerate}
\item $f \in \pcaa$ is \emph{type 1} if for every $n \in \omega$,
  $f \underline{n} \downarrow$, and there is some $m \in \omega$ such
  that $f \underline{n} = \underline{m}$
\item $e \in \pcaa$ is \emph{type 2} if for every type 1
  $f$, $e f \downarrow$ and $e f$ is type 1
\item $e \in \pcaa$ is a \emph{type 2 identity} if it is type 2
  and for all $f$ type 1 and for all $n \in \omega$, $ef
  \underline{n} = f \underline{n}$
\end{enumerate}
\end{definition}

We will now show that being able to decide whether a
term is defined or not is equivalent to the halting problem.

\begin{proposition}
\label{lambdatermprop}
  Suppose that $t(x) = t_1(x) t_2(x)$, $l \in \omega$, and $r$ is a
  normal form. If $\lmimn{l}{(\lambda x.t(x)) r } \downarrow$, then $l
  > 0$ and $\lmimn{l - 1}{t(r)} \downarrow$.
\end{proposition}

\begin{proof}
Note that from the definition of lambda terms over a pca (see
\cite{beeson85} or \cite{vanoosten}) we know that 
\[
\lambda x.t(x) := \pcas (\lambda x.t_1(x)) (\lambda x.t_2(x))
\]

Note firstly that $(\lambda x.t(x)) r = \pcas (\lambda x.t_1(x))
(\lambda x.t_2(x)) r$ and hence we can only have $\lmimn{l}{(\lambda
  x.t(x)) r} \downarrow$ for $l > 0$. Furthermore,
\[
  \lmimn{l}{\pcas (\lambda x.t_1(x)) (\lambda x.t_2(x)) r}
 \simeq  \lmimn{l - 1}{\lmimn{l - 1}{(\lambda x.t_1(x)) r} \lmimn{l
      - 1}{(\lambda x.t_2(x)) r}}
\]

Since we are assuming that $\lmimn{l}{(\lambda x.t(x)) r } \downarrow$,
we know in particular that $\lmimn{l - 1}{(\lambda x.t_1(x))
  r}\downarrow$ and $\lmimn{l - 1}{(\lambda x.t_2(x)) r}\downarrow$,
and hence
\begin{eqnarray*}
\lmimn{l - 1}{\lmimn{l - 1}{(\lambda x.t_1(x)) r} \lmimn{l
      - 1}{(\lambda x.t_2(x)) r}} 
&=& \lmimn{l - 1}{t_1(r) t_2(r)} \\
& = & \lmimn{l - 1}{t(r)}
\end{eqnarray*}
and in particular $\lmimn{l - 1}{t(r)} \downarrow$.
\end{proof}

\begin{proposition}
\label{tmprime}
For any $m,n \in \omega$, there is a closed normal form $t_m$ and a
normal form $t_{n,m}'(x)$ with free variable $x$ such that for all $r \in \term$ 
\begin{enumerate}
\item $\lmim{t_m r} \downarrow$ if and only if the $m$th Turing
  machine halts on input $m$, and if this occurs $\lmim{t_m r}
  = I$ ($I := \pcas \pcak \pcak$)
\item $\lmim{t_{n,m}'(\zeta_F) r} \downarrow$ if and only if the $m$th Turing
  machine halts on input $m$, and if this occurs $\lmim{t_m(\zeta_F) r}
  = \underline{F(n)}$
\item $t_m$ contains no $\xi_i$ and $t_{n,m}'$ contains $\xi_i$ for
  $i = n$ only
\end{enumerate}
\end{proposition}

\begin{proof}
By representability of computable functions in pcas (see eg
\cite{vanoosten} or \cite{beeson85}), one can construct
$u_m$ such that for every $k \in \omega$,
\[
u_m \underline{k} = 
\begin{cases}
\pcak I & \text{if the } m^{\text{th}} \text{ Turing machine halts
  by stage } k \text{ on input } m \\ (\lambda z.z \underline{k + 1}) & \text{if the }
m^{\text{th}} \text{ Turing machine does not halt by stage } k \text{
  on input }m \\
\end{cases}
\]

Then, following the construction in the fixed point theorem, define
\begin{eqnarray*}
w & := & \lambda x.(\lambda y. u_m y (x x)) \\
v & := & w w \\
& = & \lambda y.u_m y (w w)
\end{eqnarray*}

Then if the $m$th Turing machine halts at stage $k$ on input $m$, 
\begin{eqnarray*}
v \underline{0} & \simeq & u_m \underline{0} (w w)\\
& \simeq & u_m \underline{0} v \\
& \simeq & v \underline{1} \\
& \vdots & \\
& \simeq & v \underline{k} \\
& \simeq & u_m \underline{k} (w w) \\
& \simeq & (\pcak I) (w w) \\
& \simeq & I
\end{eqnarray*}
In particular $v \underline{0} \downarrow$.

Now suppose that the $m$th Turing machine never halts on input $m$. We
show by induction on $l$ that for all $l \in \omega$ and for all $k
\in \omega$,
\[
\lmimn{l}{v \underline{k}} \uparrow
\]

Assume that for all $k \in \omega$ and for all $l' < l$ the statement
$\lmimn{l'}{v \underline{k}} \uparrow$ holds and assume for a
contradiction that $\lmimn{l}{v \underline{k}} \downarrow$. Note that
\[
\lmimn{l}{v \underline{k}} = \lmimn{l}{((\lambda y.u_m y) (w
  w)) \underline{k}}
\]
and so by proposition \ref{lambdatermprop}, we know in particular that 
$\lmimn{l - 1}{u_m \underline{k} (w w)}\downarrow$. But in this case
\begin{eqnarray*}
\lmimn{l - 1}{u_m  \underline{k} (w w)} &=&
\lmimn{l - 2}{\lmimn{l - 2}{u_m \underline{k}} \lmimn{l - 2}{w w}} \\
& = & \lmimn{l - 2}{(\lambda z.z \underline{k + 1})
 v}  \\
& = & \lmimn{l - 3}{v \underline{k + 1}}
\end{eqnarray*}
and so in particular $\lmimn{l - 3}{v \underline{k + 1}} \downarrow$
giving a contradiction as required.

Finally, let $t_m = \pcas (\pcak v) (\pcak \underline{0})$. Then, for
all $r$, $t_m r \simeq v \underline{0}$, by the basic properties of
$\pcas$ and $\pcak$.

For parts 2 and 3, let $t_m$ be as above and let $t_{n,m}'(x) = \pcas
(\pcas t_m (\pcak x)) (\pcak \xi_n)$. Note that part 2 follows
from the basic properties of $\pcas$ and $\pcak$ and that part 3 is
clear from the definitions of $t_m$ and $t_{n,m}'(x)$.
\end{proof}

\begin{lemma}[Preservation of Atoms]
\label{presatoms}
Let $e$ be a type 2 identity in $\term$. Then for any $n$, there
is some type 1 $f$ in $\term$ such that $\lmim{e.f}$ contains the
atom $\xi_n$ as a subterm and furthermore, $f$ only contains $\xi_i$ such
that $i = n$.
\end{lemma}

\begin{proof}
  Assume that the lemma does not hold. Then there is some $n$ such
  that whenever $f$ is type 1 and contains $\xi_i$ only for $i = n$,
  we have that $\lmim{e.f}$ does not contain $\xi_n$. (That is, $e$
  ``strips away the $\xi_n$''). We will use this to derive a
  contradiction and conclude that the lemma holds.

We will define a (computable) family $f_m(x)$ of normal forms with one
free variable such that for each $F$, $f_m(\zeta_F)$ is type 1 in
$\term$.

Let $g_m \in \term$ be such that for all $l \in \omega$, $g_m
\underline{l} = \pcak (\pcak \underline{0})$ if the $m$th Turing machine
with input $m$ has not halted by stage $l$ and $g_m \underline{l} = I$ if
the $m$th Turing machine has halted by stage $l$. We can do this using
the representability of primitive recursive functions in pcas.

Then let $t'_{n,m}(x)$ be as in proposition \ref{tmprime}. Define
\[
f_m(x) := \pcas (\pcas g_m (\pcak t_{n,m}'(x))) I
\]
Note that this is in normal form and that we may assume it contains
$\xi_i$ only for $i = n$. If the $m$th Turing machine
has not halted by stage $l$ then for any $\zeta_F$
\begin{eqnarray*}
\lmim{f_m(\zeta_F) \underline{l}} & \simeq & \lmim{\lmim{(\pcas g_m (\pcak
    t_{n,m}'(\zeta_F))) \underline{l}} \underline{l}} \\
& \simeq & \lmim{\lmim{\lmim{g_m \underline{l}} t_{n,m}'(\zeta_F)} \underline{l}} \\
& \simeq & \lmim{\lmim{\pcak (\pcak \underline{0}) t_{n,m}'(\zeta_F)} \underline{l}}
  \\
& \simeq & \lmim{\pcak \underline{0} \underline{l}} \\
& \simeq & \underline{0}
\end{eqnarray*}
In particular, see that $f_m(\zeta_F) \underline{l} \downarrow$ even if
the $m$th Turing machine never halts on input $m$. If the $m$th Turing
machine on input $m$ has halted by stage $l$, then 
\begin{eqnarray*}
\lmim{f_m(\zeta_F) \underline{l}} & \simeq & \lmim{\lmim{(\pcas g_m (\pcak
    t_{n,m}'(\zeta_F))) \underline{l}} \underline{l}} \\
& \simeq & \lmim{\lmim{\lmim{g_m \underline{l}} t_{n,m}'(\zeta_F)} \underline{l}} \\
& \simeq & \lmim{\lmim{I t_{n,m}'(\zeta_F)} \underline{l}} \\
& \simeq & \lmim{t_{n,m}'(\zeta_F) \underline{l}} \\
& \simeq & \underline{F(n)}
\end{eqnarray*}

Hence for any $m \in \omega$ and any $\zeta_F$, $f_m(\zeta_F)$ is
type 1 in the sense we defined earlier. 

We therefore know that $e.f_m(\zeta_F) \simeq \lmim{e.f_m(\zeta_F)}
\downarrow$ and by hypothesis $\lmim{e.f_m(\zeta_F)}$ cannot contain
$\xi_n$. For convenience, in the below we will assume that $F$ is
chosen such that $\zeta_F$ does not occur anywhere in $e$.

Note that we can carry out an algorithm to find
$\lmim{e.f_m(\zeta_F)}$ from $m$. (Since each $F$ is the identity
everywhere except for some finite set, we can easily construct a
suitable G\"{o}del numbering for $\term$).

Furthermore, note that when we carry out this algorithm we can check
whether or not we ever need to evaluate $\lmim{t_{n,m}'(\zeta_F) r}$ for
some $r$. If we did need to evaluate this, then in particular
$\lmim{t_{n,m}'(\zeta_F) r} \downarrow$ and so the $m$th Turing machine
must halt on input $m$. On the other hand, if we did not need to
evaluate $\lmim{t_{n,m}'(\zeta_F) r}$, then $\zeta_F$ was never used in the
$\zeta$-rule because it only ever occurs as a subterm of the normal
form $t_{n,m}'(\zeta_F)$. Furthermore, by hypothesis $t_{n,m}'(\zeta_F)$ cannot
occur as a subterm of $\lmim{e.f_m(\zeta_F)}$, because otherwise
$e.f_m(\zeta_F)$ would contain $\xi_n$.

Hence if we choose $F'$ such that $F'(n) \neq F(n)$ then 
\[
\lmim{e.f_m(\zeta_F)} = \lmim{e.f_m(\zeta_{F'})}
\]
But note that this means $f_m(\zeta_F)$ and $f_m(\zeta_{F'})$ must
have the same value on $\underline{l}$ for every $l \in \omega$. This
can only happen if they are both identically zero and hence the $m$th
Turing machine does not halt on input $m$.

Therefore we could use such an algorithm to solve the halting problem
and we derive our contradiction.
\end{proof}

\subsection{Automorphisms of $\term$}
\label{automterm}
Suppose that $\pi : \omega_{>0} \rightarrow \omega_{>0}$ is a
permutation that is the identity everywhere except on some finite
set. Then $\pi$ induces an automorphism $\alpha: \term \rightarrow
\term$ as follows.

\begin{enumerate}
\item $\alpha(\xi_n) = \xi_{\pi(n)}$
\item $\alpha(\zeta_F) = \zeta_{F \circ \pi^{-1}}$
\item $\alpha(\pcas) = \pcas$
\item $\alpha(\pcak) = \pcak$
\item $\alpha(s.t) = \alpha(s).\alpha(t)$
\end{enumerate}

Note that we have chosen the action of $\alpha$ on the $\zeta_F$ so
that it is compatible with the $\zeta$-rule and the action of $\alpha$
on the $\xi_n$. $\alpha$ is clearly therefore an automorphism of
$\term$.

Note also that these automorphisms form a subgroup of the group of
automorphisms of $\term$.

\section{A Useful Lemma}
Before we move onto the proof itself, we prove a lemma that is true in
general for any pca $\pcaa$. Informally, what this says is the
property of being functional can be inherited ``up to
realizability'' across sets that are realizably equal.

\begin{lemma}
\label{iplemma}
If $a, b \in V(\pcaa)$ with $a$ functional and $V(\pcaa) \models a =
b$, then there is $e \in \pcaa$ such that whenever $\langle f, c \rangle$ and
$\langle f, c' \rangle$ are both elements of $b$,
\[
e f \Vdash c = c'
\]
\end{lemma}

\begin{proof}
Suppose that $V(\pcaa) \models a = b$ and $a$ is functional. Then there is some $e' \in \pcaa$ such that 
\[
e' \Vdash a = b
\]

Given $\langle f, c \rangle$ and $\langle f, c' \rangle$ in $b$, we
know from the definition of realizability for equality that there must be
$\langle ((e')_1 f)_0, d \rangle, \langle ((e')_1 f)_0, d' \rangle \in
a$ such that
\begin{eqnarray*}
((e')_1 f)_1 & \Vdash & c = d \\
((e')_1 f)_1 & \Vdash & c' = d'
\end{eqnarray*}

Since $a$ is functional, we know in fact that $d = d'$ and
so 
\[
\mathbf{i}_t ((e')_1 f)_1 (\mathbf{i}_s ((e')_1 f)_1) \Vdash c = c'
\]
Hence we can take
\[
e := \lambda x.\mathbf{i}_t ((e')_1 x)_1 (\mathbf{i}_s ((e')_1
x)_1)
\]
\end{proof}

\section{Failure of the Existence Property}
We will show that the existence property fails for $\czf$ in the
following instance.

\begin{theorem}
\label{noexistenceinstance}
There is no formula with one free variable $\chi(x)$ such that 
\[
\czf \vdash (\exists ! x)\,\chi(x)
\]
and 
\[
\czf \vdash \chi(x) \;\rightarrow\; x \subseteq
\mv{\mathbb{N}^{\mathbb{N}}}{\mathbb{N}} \wedge
(\forall R \in
\mv{\mathbb{N}^{\mathbb{N}}}{\mathbb{N}})(\exists S \in x)\,S \subseteq
R
\]
\end{theorem}

This will immediately give the following corollary.

\begin{corollary}
  $\czf$ does not have wEP.
\end{corollary}

\begin{proof}
  We know that
  \[ \czf \vdash (\exists x) (x \subseteq
    \mv{\mathbb{N}^{\mathbb{N}}}{\mathbb{N}} \wedge
    (\forall R \in
    \mv{\mathbb{N}^{\mathbb{N}}}{\mathbb{N}})(\exists S \in x)\,S \subseteq
    R) \]
  Suppose that there is some $\psi(x)$ such that 
  \begin{eqnarray*}
    \czf & \vdash & (\exists ! x) \,\psi(x) \\
    \czf & \vdash & (\forall x)\; \psi(x) \rightarrow (\exists z)\,z \in x \\
    \czf & \vdash & (\forall x) \;\psi(x) \rightarrow (\forall z \in x) \\
    & & \quad ( z \subseteq
    \mv{\mathbb{N}^{\mathbb{N}}}{\mathbb{N}} \wedge
    (\forall R \in
    \mv{\mathbb{N}^{\mathbb{N}}}{\mathbb{N}})(\exists S \in z)\,S \subseteq
    R)
  \end{eqnarray*}
  Then by taking $\chi(w)$ to be $(\forall x)\, (\psi(x) \rightarrow w =
  \bigcup x)$, we would get
  \[
  \czf \vdash (\exists ! w)\,\chi(w)
  \]
  and 
  \[
  \czf \vdash \chi(w) \;\rightarrow\; w \subseteq
  \mv{\mathbb{N}^{\mathbb{N}}}{\mathbb{N}} \wedge
  (\forall R \in
  \mv{\mathbb{N}^{\mathbb{N}}}{\mathbb{N}})(\exists S \in w)\,S \subseteq
  R
  \]
  contradicting the theorem.
\end{proof}

\paragraph{Proof of theorem \ref{noexistenceinstance}}

Assume that there is such a $\chi(x)$.

Let $\term$ be the pca from section \ref{sectionterm} and let $G$ be
the group of all automorphisms obtained from permutations of $\omega$
that are the identity except on some finite set, as in section
\ref{automterm}. Let $\Gamma$ be the normal filter generated by
$\{\stab{G}{\xi_n} \;|\; n \in \omega\}$. Note that $\stab{G}{x} \in
\Gamma$ exactly when there is some finite set $F \subseteq \{\xi_i
\;|\; i \in \omega \}$ such that whenever an automorphism $\alpha$
fixes every $\xi_i \in F$, $\alpha$ also fixes $x$. This is sometimes
referred to as $x$ being of \emph{finite support relative to $\{\xi_i
  \;|\; i \in \omega \}$}.

By the soundness theorems, there must be $\cip \in \vip{\term}$ and
$\cg \in \vgt$ such that
\begin{eqnarray*}
\vip{\term} & \models & \chi(\cip) \\
\vgt & \models & \chi(\cg) \\
\end{eqnarray*}

Hence we must have that 
\[
\vt \models \chi(\cip) \wedge \chi((\cg)^\circ)
\]
and so
\[
\vt \models \cip = (\cg)^\circ
\]

This allows to apply lemma \ref{iplemma} and deduce that there is some
$e_0$ such that for any $\langle f, c \rangle, \langle f, c'
\rangle \in (\cg)^\circ$,

\[
e_0.f \Vdash c = c'
\]

In fact this is the only point where we need $\cip$ and we can now
derive a contradiction by examining $\cg$ carefully.

Let $\mathbb{N}$ and $\mathbb{N}^\mathbb{N}$ be as defined in section
\ref{vaint}.  Recall that the elements of $\mathbb{N}^{\mathbb{N}}$
are of the form $\langle f, \overline{f} \rangle$ as described in
section \ref{vaint}.

Write $\zeta_1$ for $\zeta_{\lambda x.x}$ and 
for each $N$, construct $R_N \in \vgt$,
\begin{multline*}
R_N := \{ \langle 0, f, (\overline{f}, \overline{n}) \rangle \;|\; f
\text{ is type 1}, n \leq N, \underline{n} = \zeta_1 f \}\;
\cup \\
\{ \langle 0, f, (\overline{f}, \overline{n}) \rangle \;|\; f
\text{ is type 1}, n > N, \zeta_1 f = \underline{m} \text{ for some } m > N \}
\end{multline*}
(where we write $(,)$ for $\vgt$'s internal notion of ordered pairs)

\begin{lemma}
\label{rnprops}
We have constructed these $R_N$ so that the following hold:

\begin{enumerate}
\item $R_N \in \vgt$. In fact $\bigcap_{i=1}^N \stab{G}{\xi_i}
  \subseteq \stab{G}{R_N}$.
\item There is some $e_1 \in \term$ such that for all $N$, $e_1 \vdzg
  R_N \in \mvnnn$.
\item Suppose that $\langle 0, f, a \rangle \in R_N$ and $\xi_i$ occurs
  in $f$ only if $i \leq N$. Then $a = (\overline{f}, \overline{n})$
  where $\underline{n} = \zeta_1 f$ (and $n \leq N$).
\end{enumerate}
\end{lemma}

\begin{proof}
  For 1, note that each set in the binary union in the definition of
  $R_N$ is preserved by elements of $\bigcap_{i=1}^N \stab{G}{\xi_i}$.

  For 2, note that each $R_N$ can be ``represented'' by
  $\zeta_1$. This can clearly be used to produce a realizer that these
  are multi valued functions.

  Part 3 is clear from the definition.
\end{proof}

We will aim for our contradiction by first showing a lemma stating
that any automorphism satisfying certain properties has to be the
identity. This will use the key lemma from section
\ref{presatomssect} as well as the basic properties of $R_N$. We
will then construct a non trivial automorphism satisfying these
conditions. In this lemma we work over $\vt$ rather than
$\vg{\pcaa}$.

We first prove some basic properties of $V(\term)$ that will be
used in the lemma.
\begin{proposition}
  \label{vtprops}
  \begin{enumerate}
  \item for all $n, m \in \omega$, $V(\term) \models \overline{n} =
    \overline{m}$ implies that $n = m$
  \item for all $a, b \in V(\term)$, $V(\term) \models (a, b) = (a',
    b')$ if and only if $V(\term) \models a = a'$ and $V(\term)
    \models b = b'$
  \item if $g$ and $h$ are type 1 and $V(\term) \models \overline{g} =
    \overline{h}$, then for all $n \in \omega$, $g \underline{n} = h
    \underline{n}$
  \end{enumerate}
\end{proposition}

\begin{proof}
  For 1, first show that for any $n, m \in \omega$, if $V(\term)
  \models \overline{m} = \overline{n}$, then $m \geq n$. We do this by
  induction on $n$. For $n = 0$, note that we always have that $m \geq
  n$. Now suppose that $V(\term) \models \overline{m} = \overline{n +
    1}$. Note that since $\langle \underline{n}, \overline{n} \rangle
  \in \overline{n+1}$, we know there must be some $l < m$ such that
  $V(\term) \models \overline{l} = \overline{n}$. By induction, we may
  assume that $l \geq n$, and deduce that $m \geq n + 1$. By noting
  that $V(\term) \models \overline{m} = \overline{n}$ implies
  $V(\term) \models \overline{n} = \overline{m}$ we deduce the result.

  For 2, we can prove in $\czf$ the theorem that $(x, y) = (x', y')$
  if and only if $x = x'$ and $y = y'$. (See section 3.2 of
  \cite{aczelrathjen} for details). By applying the soundness theorem
  for $\czf$ we deduce that for any $a,b,a',b' \in V(\term)$,
  $V(\term) \models (a, b) = (a', b') \leftrightarrow (a = a' \wedge b
  = b')$. However, this implies the result.

  For 3, let $g$ and $h$ be of type 1, and let $n \in \omega$. Then
  there is $m \in \omega$ such that $g \underline{n} =
  \underline{m}$. By the definition of $\overline{g}$, we have
  $\langle \underline{n}, (\overline{n}, \overline{m}) \rangle \in
  \overline{g}$. By the definition of realizability for equality and
  $V(\term) \models \overline{g} = \overline{h}$, this
  implies that there is some $\langle e, c \rangle \in \overline{h}$
  such that $V(\term) \models (\overline{n}, \overline{m}) = c$. By the
  definition of $\overline{h}$ there must be $n', m' \in \omega$ such
  that $e = \underline{n'}$, $c = (\overline{n'}, \overline{m'})$ and
  $h \underline{n'} = \underline{m'}$. So $V(\term) \models
  (\overline{n}, \overline{m}) = (\overline{n'},
  \overline{m'})$. Applying the previous two parts, we get that $n =
  n'$ and $m = m'$, and deduce that $g \underline{n} = h
  \underline{n}$, as required.
\end{proof}

\begin{lemma}
\label{rnlemma}
  Suppose that $a \in V(\term)$, $N < N' \in \omega$, and $e,f \in
  \term$ are such that 

  \begin{enumerate}
  \item For any $\xi_i$ occurring in $e$ or $f$, $i \leq N$
  \item $\bigcap_{i = 1}^N \stab{G}{\xi_i} \subseteq \stab{G}{a}$
  \item $e \Vdash (\forall x \in a)\,x \in R_{N'}^\circ$
  \item $f \Vdash (\forall x \in \mathbb{N}^{\mathbb{N}})(\exists y
    \in a)(\exists z \in \mathbb{N})\,y = (x, z)$
  \end{enumerate}

  Then, whenever $\alpha \in G$ fixes $\xi_i$ for $i \leq N$ and $i >
  N'$, $\alpha$ must also fix $\xi_i$ for $N < i \leq N'$ and hence
  $\alpha$ must be the identity.
\end{lemma}

\begin{proof}
  We first check that $\lambda x.(e(f x)_0)_0 \in \term$ is a type 2
  identity.

  Let $g$ be type 1. Then $\langle g, \overline{g} \rangle \in
  \mathbb{N}^\mathbb{N}$. Therefore there is $b$ such that $\langle (f
  g)_0, b \rangle \in a$ and $(f g)_1 \Vdash (\exists z \in \mathbb{N})\,b =
  (\overline{g}, z)$ by expanding out the definition of realizability
  for bounded universal and existential quantifiers in assumption 4.

  Let $h$ be the element of $\term$ denoted by $(e (f g)_0)_0$. (Note
  that we can show the term does denote by applying the definition of
  realizability and assumptions 3 and 4).  Then we know that there is
  some $c$ such that $\langle h, c \rangle \in R_{N'}^\circ$ and $(e(f
  g)_0)_1 \Vdash b = c$, by noting that $\langle (f g)_0, b \rangle
  \in a$ and applying assumption 3. By the definition of
  $R_{N'}$ we know that $c$ must be of the form $(\overline{h},
  \overline{m})$ for some $m \in \omega$. From above we know that $\vt
  \models (\exists z \in \mathbb{N})\,b = (\overline{g}, z)$ and $\vt
  \models b = (\overline{h}, \overline{m})$, so we deduce that $\vt
  \models \overline{g} = \overline{h}$ by applying part 2 of
  proposition \ref{vtprops}. Hence $g$ and $h$ must have the same
  graphs as type 1 elements by part 3 of proposition \ref{vtprops}, and
  so since $(\lambda x.(e (f x)_0)_0) g$ denotes $h$, $\lambda x.(e (f
  x)_0)_0$ must be a type 2 identity as required.

  Now let $\alpha \in G$ fix $\xi_i$ for $i \leq N$ and $i >
  N'$. Suppose for a contradiction that there is some $n$ with $N < n
  < N'$ such that $\alpha(\xi_n) \neq \xi_n$. 

  By applying lemma \ref{presatoms} we can find a type 1 $g \in \term$
  such that $g$ only contains $\xi_i$ for $i = n$ and such that
  $\xi_n$ does occur in the element of $\term$ denoted by $(e (f
  g)_0)_0$ (or equivalently by proposition \ref{termdefsequiv} the
  result of reducing $(e (f g)_0)_0$ by leftmost innermost
  reduction). Let $b$, $h$, and $c$ be as above, but for this
  particular $g$.

  Since $e$ and $f$ only contain $\xi_i$ for $i \leq N$ and $g$ only
  contains $\xi_n$, we know that
  $h$ can only contain $\xi_i$ for $i \leq N$ or $i = n$, because
  applying leftmost innermost reduction can't introduce any new $\xi_i$
  that weren't already in the term $(e (f g)_0)_0$. Since we
  have guaranteed that $h$ does contain $\xi_n$, we know that $\zeta_1
  h = \underline{n}$.  In particular $n \leq N'$, so we know from part
  3 of lemma \ref{rnprops} that
  $c$ must be of the form $(\overline{h},
  \overline{n})$.

  Since $\alpha$ fixes $\xi_i$ for $i \leq N$ we know from 
  assumption 2 that $\alpha$ also fixes $a$. Therefore, since $\langle
  (f g)_0, b \rangle \in a$ we must also have $\langle \alpha((f g)_0),
  \alpha(b) \rangle \in a$. Hence if $h'$ is the element of $\term$
  denoted by $(e \alpha((f g)_0))_0$ we
  know that there is some $c'$ with $\langle h', c' \rangle \in
  R_{N'}^\circ$ and $\vt \models \alpha(b) = c'$ by applying
  assumption 3 and the definition of $R_{N'}$ as we did before.

  Since $\alpha$ fixes $\xi_i$ for $i \leq N$ we know that $\xi_i$ can
  only occur in $e$ and $\alpha(f)$ for $i \leq N$. Since $\alpha$
  fixes $\xi_i$ for $i > N'$, we know that $\alpha(\xi_n)$ must be
  some $\xi_i$ where $i \leq N'$. Hence $h'$ only contains $\xi_i$
  for $i \leq N'$. Furthermore neither $\alpha(f)$ nor $\alpha(g)$
  contains $\xi_n$ and from the assumption that $\alpha(\xi_n) \neq
  \xi_n$ we also know that $\xi_n$ does not occur in
  $\alpha(g)$. Hence $\zeta_1 h' = \underline{m}$ for some $m \leq N'$
  with $m \neq n$. Again from part 3 of lemma \ref{rnprops}, we know
  therefore that $c'$ is of the form $(\overline{h'},
  \overline{m})$. In particular, since $V(\term) \models \alpha(b) =
  c'$, this means that $V(\term) \models
  \alpha(b) = (\overline{h'}, \overline{m})$.

  But then since $\vt \models b = ( \overline{h}, \overline{n} )$, we
  have that $\vt \models \alpha(b) = (\alpha(\overline{h}),
  \alpha(\overline{n}))$ by proposition
  \ref{automspresrealiz}. Together with $\vt \models \alpha(b) =
  (\overline{h'}, \overline{m})$ and applying part 2 of proposition
  \ref{vtprops}, this gives $\vt \models \alpha(\overline{n}) =
  \overline{m}$. In fact $\alpha(\overline{n}) = \overline{n}$, since
  $\overline{\omega}$ was constructed only using elements of $\term$
  fixed by all automorphisms.  So $\vt \models \overline{n} =
  \overline{m}$, and applying part 1 of proposition \ref{vtprops} we
  get $n = m$. But this is a contradiction since $m \neq n$.
\end{proof}

Since we know $\vgt \models (\forall x \in \mvnnn)(\exists y \in \cg)\,( y
\subseteq x \,\wedge\, y \in \mvnnn)$, there must be $f, e_2, e_3 \in
\term$ and $c_n$ such that for all $n$,
\begin{eqnarray*}
\langle 0, f, c_n \rangle & \in & \cg \\
e_2 & \vdzg & (\forall x \in c_n)\,x \in R_n \\
e_3 & \vdzg & (\forall x \in \mathbb{N}^{\mathbb{N}})(\exists y \in
c_n)(\exists z \in \mathbb{N})\,y = (x, z)
\end{eqnarray*}

In particular we know that for all $n$, $\stab{G}{c_n} \in \Gamma$ and
hence by proposition \ref{sympreserved} $\stab{G}{c_n^\circ } \in
\Gamma$. From now on we will work entirely over $\vt$.

Recall that we chose $e_0$ so that for all $m$ and $n$,
\[
e_0 f \Vdash c_m^\circ = c_n^\circ
\]
and so by substitution we can use $e_0 f$ and $e_2$ to construct $e_4$
such that for all $m$ and $n$
\[
e_4 \Vdash (\forall x \in c_m^\circ)\,x \in R_n^\circ
\]

Now let $N$ be large enough such that the list $\xi_1,\ldots,\xi_N$
includes any $\xi_n$ in a support of $c_0^\circ$, or appearing in $e_0, e_1
e_2, e_3$ or $e_4$.

Let $N' = N + 2$.

Note that we have
\[
e_4 \Vdash (\forall x \in c_0^\circ)\,x \in R_{N'}^\circ
\]

Let $\alpha$ be the automorphism that swaps round $\xi_{N + 1}$ and
$\xi_{N + 2}$, fixing everything else. Then we know that $\alpha$
fixes $\xi_i$ for $i \leq N$ (and hence also fixes $c_0^\circ$ and any
$\xi_i$ occurring in $e_3$ and $e_4$) and fixes $\xi_i$ for $i >
N'$. However, clearly $\alpha$ does not fix $\xi_{N +
  1}$. Hence we can finally get a contradiction by applying lemma
\ref{rnlemma}.

\qed

\section{Conclusion}
\subsection{What the failure of EP means}
We have shown that $\czf$ does not have EP, or indeed even wEP. Since
EP was described in the introduction as a property to be expected from
constructive formal theories, one might ask if its failure indicates
some weakness in $\czf$ as a constructive theory. The
short answer is no: $\czf$ is still a sound foundation for
constructive mathematics. 

The main theorem of this paper shows essentially that $\czf$ asserts
the existence of mathematical objects that it does not know how to
construct. However, $\czf$ does have natural interpretations in which
these objects can be constructed. One example is Aczel's original
interpretation of $\czf$ into type theory in \cite{Aczel78}. Here, the
sets asserted in the fullness axiom are sets of those multivalued
functions that arise from elements of a particular exponential
type. Another (related) interpretation is Rathjen's ``formulas as
classes'' in \cite{formulaeasclasses}, in which $\czf$ is interpreted
into $\czfexp$. In this example the full sets appear as exponentials
in the background universe. In \cite{rathjentupailo06} Rathjen and
Tupailo showed using these techniques that $\czf$ with a choice
principle $\mathbf{\Pi \Sigma - AC}$ has a form of the existence
property.

\subsection{Further Work}
In this paper we used the axioms of $\zfc$ in several places, with
what appears to be an essential use of choice in the soundness of
strong collection in $\vip{\pcaa}$. This means that the final result
of the paper was only proved on the assumption that $\zfc$ is
consistent. We conjecture that in fact this assumption is unwarranted
and with a more sophisticated construction the result can shown only
on the assumption that $\czf$ is consistent.

Choice principles tend to fail in $\vg{\pcaa}$ (in fact one can
check that countable choice fails in $\vgt$), so it remains open
whether $\czf$ together with particular choice principles have EP.

\section{Acknowledgements}
\label{sec:acknowledgements}

I would like to thank my PhD supervisor, Michael Rathjen for helping
me greatly in preparing this paper. Peter Aczel also provided some
helpful comments. This work was supported by an EPSRC Doctoral
Training Grant. I would also like to thank the anonymous referee, who
made several useful suggestions for improving the paper.

\bibliographystyle{abbrv}
\bibliography{realizability}{}

\begin{thebibliography}{10}

\bibitem{Aczel78}
P.~Aczel.
\newblock The type theoretic interpretation of constructive set theory.
\newblock In A.~MacIntyre, L.~Pacholski, and J.~Paris, editors, {\em Logic
  Colloquium '77}, pages 55--66. North--Holland, Amsterdam-New York, 1978.

\bibitem{aczelrathjen}
P.~Aczel and M.~Rathjen.
\newblock Notes on constructive set theory.
\newblock Technical Report~40, Institut Mittag-Leffler, 2001.

\bibitem{beeson85}
M.~Beeson.
\newblock {\em Foundations of Constructive Mathematics: Metamathematical
  Studies}.
\newblock Springer, 1985.

\bibitem{BeesonScedrov}
M.~Beeson and A.~\u{S}\u{c}edrov.
\newblock Church's thesis, continuity, and set theory.
\newblock {\em Journal of Symbolic Logic}, 49:630--643, 1984.

\bibitem{bvmodels}
J.~L. Bell.
\newblock {\em Boolean-Valued Models and Independence Proofs in Set Theory},
  volume~12 of {\em Oxford Logic Guides}.
\newblock Oxford University Press, 1985.

\bibitem{friedman73a}
H.~Friedman.
\newblock The consistency of classical set theory relative to a set theory with
  intuitionistic logic.
\newblock {\em Journal of Symbolic Logic}, 38:315--319, 1973.

\bibitem{friedman73}
H.~Friedman.
\newblock Some applications of {K}leene's method for intuitionistic systems.
\newblock In A.~Mathias and H.~Rogers, editors, {\em Cambridge Summer School in
  Mathematical Logic}, volume 337 of {\em Lecture Notes in Mathematics}, pages
  113--170. Springer, Berlin, 1973.

\bibitem{FriedmanScedrov83}
H.~Friedman and A.~\u{S}\u{c}edrov.
\newblock Set existence property for intuitionistic theories with \ dependent
  choice.
\newblock {\em Annals of Pure and Applied Logic}, 25:129--140, 1983.

\bibitem{FriedmanScedrov85}
H.~Friedman and A.~\u{S}\u{c}edrov.
\newblock The lack of definable witnesses and provably recursive \ functions in
  intuitionistic set theory.
\newblock {\em Advances in Mathematics}, 57:1--13, 1985.

\bibitem{jechac}
T.~J. Jech.
\newblock {\em The Axiom of Choice}.
\newblock North-Holland, 1973.

\bibitem{jechst}
T.~J. Jech.
\newblock {\em Set Theory}.
\newblock Springer-Verlag, 2002.

\bibitem{kleene45}
S.~C. Kleene.
\newblock On the intepretation of intuitionistic number theory.
\newblock {\em Journal of Symbolic Logic}, 10:109--124, 1945.

\bibitem{kreiseltroelstra}
G.~Kreisel and A.~S. Troelstra.
\newblock Formal systems for some branches of intuitionistic analysis.
\newblock {\em Annals of Mathematical Logic}, 1:229--387, 1970.

\bibitem{kunen}
K.~Kunen.
\newblock {\em Set Theory: An Introduction to Independence Proofs}.
\newblock North-Holland, 1980.

\bibitem{LubarskyRathjen}
R.~Lubarsky and M.~Rathjen.
\newblock On the constructive {D}edekind reals.
\newblock {\em Logic and Analysis}, 1(2):131--152, 2008.

\bibitem{mccarty}
D.~C. McCarty.
\newblock {\em Realizability and Recursive Mathematics}.
\newblock PhD thesis, Oxford University, 1984.

\bibitem{mccarty86}
D.~C. McCarty.
\newblock Realizability and recursive set theory.
\newblock {\em Annals of Pure and Applied Logic}, 32:153--183, 1986.

\bibitem{Myhill73}
J.~Myhill.
\newblock Some properties of intuitionistic {Z}ermelo--{F}raenkel \ set theory.
\newblock In A.~R.~D. Mathias and H.~Rogers, editors, {\em Cambridge Summer
  School in Mathematical Logic}, volume 337 of {\em Lecture Notes in
  Mathematics}, pages 206--231. Springer, Berlin, 1973.

\bibitem{Myhill75}
J.~Myhill.
\newblock Constructive set theory.
\newblock {\em Journal of Symbolic Logic}, 40:347--382, 1975.

\bibitem{rathjen05}
M.~Rathjen.
\newblock The disjunction and other properties for {C}onstructive
  {Z}ermelo-{F}rankel set theory.
\newblock {\em Journal of Symbolic Logic}, 70:1233--1254, 2005.

\bibitem{formulaeasclasses}
M.~Rathjen.
\newblock The formulae-as-classes interpretation of constructive set theory.
\newblock In H.~Schwichtenberg and K.~Spies, editors, {\em Proof Technology and
  Computation}, volume 200 of {\em Series III: Computer and Systems Sciences},
  pages 279 -- 322. IOS Press, 2006.

\bibitem{rathjen06}
M.~Rathjen.
\newblock Realizability for constructive {Z}ermelo-{F}raenkel set theory.
\newblock In V.~Stoltenberg-Hansen and J.~V\"{a}\"{a}n\"{a}nen, editors, {\em
  Logic Colloquium '03}, volume 141 of {\em Lecture Notes in Logic}, pages
  442--471. Association for Symbolic Logic, 2006.

\bibitem{rathjen08}
M.~Rathjen.
\newblock Metamathematical properties of intuitionistic set theories with
  choice principles.
\newblock In C.~S. Barry, B.~L\"{o}we, and S.~Andrea, editors, {\em New
  Computational Paradigms}, pages 84--98. Springer, 2008.

\bibitem{rathjen11}
M.~Rathjen.
\newblock From the weak to the strong existence property.
\newblock {\em Annals of Pure and Applied Logic}, 163(10):1400 -- 1418, 2012.

\bibitem{rathjentupailo06}
M.~Rathjen and S.~Tupailo.
\newblock Characterizing the interpretation of set theory in {M}artin-{L}\"{o}f
  type theory.
\newblock {\em Annals of Pure and Applied Logic}, 141:253--258, 2006.

\bibitem{troelstravandalen}
A.~S. Troelstra and D.~van Dalen.
\newblock {\em Constructivism in Mathematics: An Introduction}, volume 121 of
  {\em Studies in Logic and the Foundation of Mathematics}.
\newblock North-Holland, 1988.

\bibitem{vanoosten}
J.~van Oosten.
\newblock {\em Realizability: An Introduction to its Categorical Side}, volume
  152 of {\em Studies in Logic and the Foundations of Mathematics}.
\newblock Elsevier, 2008.

\end{thebibliography}

\end{document}